\documentclass[11pt]{article}

\usepackage{setspace}
\usepackage[english]{babel}
\usepackage[utf8]{inputenc}
\usepackage{amsmath}
\usepackage{scalerel}
\usepackage{amssymb}
\usepackage{graphicx}
\usepackage{textcomp}
\usepackage{fancyhdr}
\usepackage{amsthm}
\usepackage{amssymb}
\usepackage{graphicx}
\usepackage{appendix}
\usepackage{indentfirst}
\usepackage{mathtools}
\usepackage{bm}
\usepackage{biblatex}
\renewbibmacro{in:}{}
\usepackage{enumerate}
\usepackage{mathrsfs}
\mathtoolsset{showonlyrefs,showmanualtags}
\usepackage{hyperref}
\usepackage{csquotes}
\usepackage[margin=0.9in]{geometry}
\allowdisplaybreaks
\numberwithin{equation}{section}
\newcommand{\lt}{\lfloor T\rfloor}
\newcommand{\Prob}{\bm{\mathrm{P}}}
\newcommand{\e}{\varepsilon}

\usepackage[active]{srcltx}

\newtheorem{theorem}{Theorem}[section]

\newtheorem{lemma}[theorem]{Lemma}
\newtheorem{proposition}[theorem]{Proposition}
\newtheorem{definition}[theorem]{Definition}
\newtheorem{example}[theorem]{Example}

\newtheorem{remark}[theorem]{Remark}

\newtheoremstyle{named}{}{}{\itshape}{}{\bfseries}{.}{.5em}{\thmnote{#3}#1}
\theoremstyle{named}
\newtheorem*{namedtheorem}{}

\newcommand{\E}{\bm{\mathrm E}}
\hypersetup{
    colorlinks=true,
    linkcolor=black,
    citecolor=black,
    pdftitle={Local limit theorem for time-inhomogeneous functions of Markov processes},
    }
\title{Local limit theorem for time-inhomogeneous functions of Markov processes}
\author{Leonid Koralov\footnote{Department of Mathematics, University of Maryland, Email: koralov@umd.edu},\quad Shuo Yan\footnote{Department of Mathematics, Imperial College London, Email: s.yan@imperial.ac.uk}}
\date{}

\begin{document}

\maketitle
\begin{abstract}
    In this paper, we consider a continuous-time Markov process and prove a local limit theorem for the integral of a time-inhomogeneous function of the process. One application is in the study of the fast-oscillating perturbations of linear dynamical systems.
    

\end{abstract}
\section{Introduction}
\label{sec:intro}
Suppose that $\{X_s^x\}_{s\geq0}$ is a time-homogeneous Markov family on a metric space $(\bm X,\mathscr B(\bm X))$ with $P(s,x,\Gamma)$ being its transition function. 
We will drop the superscript $x$ from the notation when it is clear what the initial point or distribution of the process is.
Let $\mathcal B$ be the Banach space of bounded $\mathscr B(\bm X)$-measurable complex functions on $\bm X$ equipped with the supremum norm: $\|f\|=\sup_{x\in\bm X}|f(x)|$.
Then we have the corresponding semigroup $\{P_s\}_{s\geq0}$ defined as
\begin{equation*}
    P_sf(x)=\E_xf(X_s)=\int_{\bm X} f(y)P(s,x,dy).
\end{equation*}
In order to state the local limit theorem, we make additional assumptions about $\{X_s\}_{s\geq0}$. 
Among all, the most essential condition is the uniqueness of the invariant measure and the uniform convergence of $X_s$ to this measure, which can be re-formulated in terms of the quasi-compactness (\cite{HennionHerve}) of the operators $P_s$, $s>0$. To be more precise, an operator $P$ acting on $\mathcal B$ is said to satisfy the property of quasi-compactness if $\mathcal B$ can be decomposed into two $P$-invariant closed subspaces: 
\begin{equation}
\label{eq:quasi-compact}
    \mathcal B=F\oplus H,
\end{equation}
where 
$r({P|}_{H})<r(P)$, while dim$(F)<\infty$, and each eigenvalue of ${P}_{|F}$ has modulus $r(P)$. 
Here, $r(\cdot)$ stands for the spectral radius of an operator. 
Suppose that $\nu$ is the unique invariant measure of $\{X_s\}_{s\geq0}$, and the distribution of $X_s$ converges to $\nu$ in total variation, uniformly in initial distribution. 
Then it is clear that, for every $s>0$, $P_s$ is quasi-compact with $F=\mathrm{span}\{1\}$ and $H=\{f\in\mathcal B:\langle\nu,f\rangle=0\}$, where $\nu$ is understood as a functional on $\mathcal B$ via $\langle\nu,f\rangle=\int_{\bm X}fd\nu$. 
A more direct way to understand quasi-compactness is to give a decomposition of the operator itself. 
Define $\mathcal L_{\mathcal B}$ to be the space of all bounded linear operators from $\mathcal B$ to itself, and $\mathcal B'$ to be the space of all bounded linear operators from $\mathcal B$ to $\mathbb C$.
For $\varphi\in\mathcal B'$ and $v\in\mathcal B$, we define
$v\otimes\varphi\in\mathcal L_{\mathcal B}$ by $(v\otimes\varphi)f=\langle\varphi,f\rangle v$ for all $f\in\mathcal B$. 
Let us denote $Q=P_1$ be the operator corresponding to the Markov process $\{X_s\}_{s\geq0}$ sampled at integer points.
As we discussed, $Q$ is quasi-compact with $F=\mathrm{span}\{1\}$ and $H=\{f\in\mathcal B:\langle\nu, f\rangle=0\}$. So, we have the following decomposition for $Q$:
    \begin{equation}
    \label{eq:multi-decomposition}
        Q=1\otimes\nu +N,
    \end{equation}
    where $N=Q\Pi_H$, $\Pi_H$ is the projection onto $H$, and $r(N)<1$.

In applications, one often considers limit theorems for additive functionals of Markov processes (\cite{HennionHerve},\cite{Hitsuda},\cite{spectralgap}). 
For example, let $b(x)$ be a real-valued bounded continuous function on $\bm X$.
One is interested in describing the distribution of $\int_0^T b(X_s)ds$ as $T\to\infty$. 
So, in addition to the process $\{X_s\}_{s\geq0}$, assumptions also need to be made about the function $b(x)$. 
Similar questions have also been studied for discrete-time Markov chains, in which case, one is interested in the distribution of the sum $\sum_{k=1}^n b(X_k)$ as $n\to\infty$.

In this paper, we focus on the integral formulation of the problem as described but allow the function $b$ to vary in time and be vector-valued.  
Namely, let $b:[0,1]\times\bm X\to\mathbb R^d$, $d\in\mathbb N$ and $S_T=\int_0^T b(s/T,X_s)ds$. 
We will prove that a local limit theorem holds for $S_T$. 
The spectral method has been applied to prove various limit theorems for Markov chains (cf. \cite{HennionHerve}, \cite{BSMF_2010__138_3_415_0}, \cite{MR1875670}, \cite{spectralgap}, \cite{MR3793181}, and the references therein).
For example, in the case where the function $b$ does not depend on time, the local limit theorem can be obtained by studying the characteristic function of the sum (or the integral in the continuous case) or, equivalently, the limiting behavior of the power of the corresponding Fourier kernels.
Due to the quasi-compactness of the corresponding operator, the limiting behavior of its power is primarily determined by the dominant eigenvalues and the corresponding eigenfunctions.
However, in this paper, we need to deal with the additional dependence on another parameter that appears in the discretization of the time interval $[0,T]$ due to the time-inhomogeneity of the function $b$, and study the product of a sequence of different Fourier operators that depend on the parameter.
This main technical step is given in Proposition~\ref{prop:multi-operator_product}.

The paper is organized as follows: In the rest of Section~\ref{sec:intro}, we introduce some notation, formulate the assumptions and the main result, and discuss some applications. In Section~\ref{sec:prelim}, we give preliminary results to motivate the ingredients of the proof. In Section~\ref{sec:product}, we prove the technical results. In Section~\ref{sec:proof}, we give the proof of our main result, Theorem~\ref{thm:main_result}.

Throughout this article, $\Prob$ and $\E$ represent the probability and expectation, respectively, and the  subscripts pertain to initial conditions. 
The Lebesgue measure is denoted by $\mathfrak L$.
For each $k\geq 0$, a subset $X$ of a Euclidean space, and a Banach space $Y$, $\mathcal C^k(X,Y)$ denotes the space of $k$ times continuously differentiable functions from $X$ to $Y$. 
For $C>0$, let $\mathcal B_{C}=\{f\in\mathcal B:\|f\|\leq C\}$.
Let $\mathcal B'_p$ be the space of probability measures on $\bm X$.
Let $\lfloor T\rfloor$ denote the integer part of $T$.

\begin{definition}
    The operator $Q(t,\alpha,\beta)\in\mathcal L_{\mathcal B}$, $t\in\mathbb R^d$, $\alpha,\beta\in[0,1]$, is defined as
    \begin{equation}
    \label{def:Q(t,alpha,beta)}
        Q(t,\alpha,\beta)f(x)=\E_x\left[\exp\left(it\cdot\int_0^1b\left(\alpha+s(\beta-\alpha),X_s\right)ds\right)f(X_1)\right].
    \end{equation}
    In particular, for each $\alpha$ and $\beta$, and $t=0$, this operator coincides with $Q$.
\end{definition}

\textbf{\textit{Assumptions:}}
\begin{enumerate}
    \item[\hypertarget{as:markov_process}{(A1)}] $\{X_s\}_{s\geq0}$ is a Markov process on the complete separable state space $(\bm X,\mathscr B(\bm X))$ with sample paths in the Skorokhod space $D([0,\infty),\bm X)$.
    \item[\hypertarget{as:mixing}{(A2)}] $\nu$ is the unique invariant measure of $\{X_s\}_{s\geq0}$ on $\bm X$. There exists $\mathcal T>0$ such that \begin{equation*}
        \sup_{x\in\bm X}\mathrm{TV}(P(\mathcal T,x,\cdot),\nu)<1,
    \end{equation*}
    where $TV$ is the total variation distance between two probability measures.
    \item[\hypertarget{as:zero_mean}{(A3)}] 
    $b:[0,1]\times\bm X\to\mathbb R^d$ is continuous and bounded.
    For each $0\leq\alpha\leq 1$, $\int_{\bm X}b(\alpha,\cdot)d\nu=0$.
    For each $x\in\bm X$, $b(\alpha,x)$ is twice differentiable in $\alpha$.
    Moreover, $\frac{\partial}{\partial\alpha}b(\alpha,x)$ and $ \frac{\partial^2}{\partial\alpha^2}b(\alpha,x)$ are continuous and bounded on $[0,1]\times\bm X$. 
    \item[\hypertarget{as:non-arithmetic}{(A4)}] Non-arithmetic condition on $(\{X_s\}_{s\geq0},b(\alpha,x))$ holds: For every $0\leq\alpha\leq1$ and every $t\not=0$, $r(Q(t,\alpha,\alpha))<1$.
\end{enumerate}
\begin{remark}
It is not hard to verify that,
under Assumption~\hyperlink{as:mixing}{(A2)}, $Q$ is quasi-compact as defined in \eqref{eq:quasi-compact} and has the decomposition \eqref{eq:multi-decomposition}.
In addition, under Assumption~\hyperlink{as:zero_mean}{(A3)}, $Q(t,\alpha,\beta)\in\mathcal C^2(\mathbb R^d\times[0,1]^2,\mathcal L_{\mathcal B})$.
\end{remark}
\begin{remark}
    The non-arithmetic condition \hyperlink{as:non-arithmetic}{(A4)} is required to prove the local limit theorem. 
    However, the formulation is not intuitive and often not easy to verify.
    In fact, with Assumptions \hyperlink{as:markov_process}{(A1)-(A3)}, it can be reduced to a more natural form - the non-lattice condition:
    \begin{enumerate}
        \item[\hypertarget{as:non-lattice}{\textup{(NL)}}] For each $\alpha\in[0,1]$, there is no $a\in\mathbb R^d$, a closed subgroup $H\subsetneqq \mathbb R^d$, and bounded measurable function $u:\bm X\to\mathbb R^d$ such that $\int_0^1 b(\alpha,X_s)ds+u(X_1)-u(X_0)\in a+H$, $\nu$-a.s.
    \end{enumerate}
    In the Appendix, we will provide a proof of this statement.
\end{remark}
\begin{theorem}
\label{thm:main_result}
    Let Assumptions \hyperlink{as:markov_process}{(A1)-(A4)} be satisfied and 
    \begin{equation}
    \label{eq:def_of_S}
        S_T=\int_0^T b(s/T,X_s)ds.
    \end{equation}
    Then there exists a positive-definite matrix $\Sigma$ such that, for each compactly supported function $g\in \mathcal C(\mathbb R^d,\mathbb R)$, and $C>0$, uniformly in real-valued $f\in\mathcal B_C$, initial distribution $\mu\in\mathcal B'_p$, and $u\in\mathbb R^d$,
    \begin{equation*}
        \lim_{T\to\infty}|\mathrm{det}(\Sigma)(2\pi T)^{d/2}\E_\mu[f(X_T)g(S_T-u)]-e^{-\frac{1}{2T}u^*(\Sigma\Sigma^*)^{-1}u}\langle\nu,f\rangle\langle \mathfrak L,g\rangle|=0.
    \end{equation*}
\end{theorem}
\begin{example}
    The assumptions are satisfied if $X_s$ is a non-degenerate diffusion process on $\mathbb T^m$, \hyperlink{as:zero_mean}{(A3)} is satisfied, and $\{b(\alpha,x)$, $x\in\bm X\}$ spans $\mathbb R^d$ for each $0\leq \alpha\leq1$.
\end{example}

Before starting the proof of Theorem~\ref{thm:main_result}, let us discuss an application to the averaging principle for randomly perturbed linear dynamical systems.
\begin{theorem}
\label{thm:linear_dynamics}
    Let $Y_t^\e$ be defined by the equation:
    \begin{equation}
    \label{eq:random_linear_dynamics}
        dY_t^\e=[A(t)Y_t^\e+v(t,X_{t/\e})]dt,~~Y_0^\e=y\in\mathbb R^d,
    \end{equation}where $A(t)$ is a continuously differentiable $d$-dimensional square-matrix-valued function, $v:[0,\infty)\times\bm X\to\mathbb R^d$ is continuous, and $\{X_s\}_{s\geq0}$ satisfies  Assumptions \hyperlink{as:markov_process}{(A1)} and \hyperlink{as:mixing}{(A2)}.
    Let $\int_{\bm X}v(t,x)d\nu (x)=\bar v(t)$ and $y_t$ be the averaged motion defined by:
    \begin{equation}
    \label{eq:linear_dynamics}
        dy_t=[A(t)y_t+\bar v(t)]dt,~~y_0=y.
    \end{equation} Fix $t>0$ and suppose that $v(\alpha t,x)-\bar v(\alpha t)$ as a function of $\alpha$ and $x$ satisfies \hyperlink{as:zero_mean}{(A3)} and the non-arithmetic condition \hyperlink{as:non-arithmetic}{(A4)} holds for $(\{X_s\}_{s\geq0},v(\alpha t,x)-\bar v(\alpha t))$.
    Then the local limit theorem holds for $\frac{1}{\e}(Y_t^\e-y_t)$ as $\e\to\infty$.
    Namely, there exists a positive-definite matrix $\Sigma_{t}$ such that, for each compactly supported continuous real-valued function $g$, we have
    \begin{equation*}
        \lim_{\e\to0}\left|\mathrm{det}(\Sigma_{t})(2\pi t /\e)^{d/2}\E_\mu\left[f(X_{t/\e})g\left(\frac{1}{\e}(Y_t^\e-y_t)-u\right)\right]-e^{-\frac{1}{2T}u^*(\Sigma_{t}\Sigma_{t}^*)^{-1}u}\langle\nu,f\rangle\langle \mathfrak L,g\rangle\right|=0,
    \end{equation*}
    uniformly in real-valued measurable $f$ satisfying $\|f\|\leq C$, $u\in\mathbb R^d$, and probability measure $\mu$. 
    
\end{theorem}
    \begin{proof}
    By taking the difference of \eqref{eq:random_linear_dynamics} and \eqref{eq:linear_dynamics} we obtain
    \begin{align}
    \label{eq:inhomogeneous_equation}
        \frac{1}{\e}(Y_t^\e-y_t)&=\frac{1}{\e}\int_0^t\left(A(s)(Y_s^\e-y_s)+v(s,X_{s/\e})-\bar v(s)\right)ds.
    \end{align}
    Let $U(t)$ solve the equation:
    \begin{equation*}
        \begin{cases}
            dU(t)=A(t)U(t)dt\\
            U(0)=I
        \end{cases}.
    \end{equation*}
    Then, by Duhamel's principle,
    \begin{equation}
    \label{eq:duhamel}
        \frac{1}{\e}(Y_t^\e-y_t)=\frac{1}{\e}\int_0^t U(t)U(s)^{-1}\left(v(s,X_{s/\e})-\bar v(s)\right)ds.
    \end{equation}(One can also verify directly that \eqref{eq:duhamel} holds by differentiating both sides and using \eqref{eq:random_linear_dynamics} and \eqref{eq:linear_dynamics}.)
    After the change of variable, with $T=t/\e$, we obtain that 
    \begin{equation*}
    \begin{aligned}
        \frac{1}{\e}(Y_t^\e-y_t)&=\int_0^T U(t)U({st}/{T})^{-1}\left(v(st/T,X_{s})-\bar v(st/T)\right)ds.
    \end{aligned}
    \end{equation*}
    The integral has the form of \eqref{eq:def_of_S} with $b(\alpha,x)=U(t)U(\alpha t)^{-1}\left(v(\alpha t,x)-\bar v(\alpha t)\right)$. 
    It is not hard to see that Assumption \hyperlink{as:non-arithmetic}{(A4)} holds for $(\{X_s\}_{s\geq0},b(\alpha,x))$ since $U(\cdot)$ is non-degenerate.
    Indeed, the operator $Q(t_1,\alpha,\beta)$ in \eqref{def:Q(t,alpha,beta)} corresponding to $(\{X_s\}_{s\geq0},b(\alpha,x))$ with $t_1\not=0$ is the same as the operator $Q(t_2,\alpha,\beta)$ corresponding to $(\{X_s\}_{s\geq0},v(\alpha t,x)-\bar v(\alpha t))$, where $t_2=U^{-1}(\alpha t)^*U(t)^*t_1\not=0$ and $*$ stands for transposition without taking the complex conjugate. This operator has its spectral radius less than $1$ by the assumption.
    Here, to avoid confusion, let us stress that $t_1$ and $t_2$ are parameters in \eqref{def:Q(t,alpha,beta)}, while $t$ is our fixed time.
    Finally, Assumption \hyperlink{as:zero-mean}{(A3)} follows from the assumptions we have on functions $A$ and $v$.
    Thus, the result follows from Theorem~\ref{thm:main_result}.
\end{proof}
\begin{remark}
    The functional central limit theorem is well-known for the fast-slow systems of the form:
    \begin{equation}
    \label{eq:fast-slow}
        dZ_t^\e=c(Z_t^\e,X_{t/\e})dt.
    \end{equation}
    Namely, under natural assumptions, $Z_t^\e$ converges to a deterministic process $z_t$ in probability as $\e\downarrow0$ on each finite interval, and $\frac{1}{\sqrt{\e}}(Z_t^\e-z_t)$ converges to a Gaussian Markov process weakly (\cite{MR0203789}, \cite{MR0517995}).
    The functional central limit theorem gives, in particular, the distribution of $Z_t^\e$ in a neighborhood of $z_t$ of size $O(\sqrt{\e})$.
    Theorem~\ref{thm:linear_dynamics} refines this statement to spatial scales of order $\e$ in the special case of linear systems.
\end{remark}
In addition, it is not hard to verify that the arguments in this paper can be applied to prove the following slightly stronger version of the local limit theorem:
\begin{theorem}
\label{thm:main_result_proved}
    Let assumptions \hyperlink{as:markov_process}{(A1)-(A4)} be satisfied.  For $0\leq\rho<1$, define $S(\rho,T)=\int_{0}^{(1-\rho)T} b(\rho+s/T,X_{s})ds$. Then there exists a positive-definite matrix $\Sigma_\rho$ that is continuous in $\rho$ such that, for each compactly supported function $g\in \mathcal C(\mathbb R^d,\mathbb R)$, $C>0$, and $0\leq\rho_0<1$, uniformly in real-valued $f\in\mathcal B_C$, initial distribution $\mu\in\mathcal B'_p$, $0\leq\rho\leq\rho_0$, and $u\in\mathbb R^d$,
    \begin{equation*}
        \lim_{T\to\infty}\left|\mathrm{det}(\Sigma_\rho)(2\pi T)^{d/2}\E_\mu[f(X_{T(1-\rho)})g(S(\rho,T)-u)]-e^{-\frac{1}{2T}u^*(\Sigma_\rho\Sigma_\rho^*)^{-1}u}\langle\nu,f\rangle \langle \mathfrak L,g\rangle\right|=0.
    \end{equation*}
\end{theorem}
\begin{remark}
    The only extra work needed in the proof of Theorem~\ref{thm:main_result_proved} is to keep track of the dependence of all terms on $\rho$ and make sure that all the inequalities hold uniformly for all $\rho\in[0,\rho_0]$ where $0\leq\rho_0<1$.
    The formulation of the result in Theorem~\ref{thm:main_result_proved} might seem strange at first. However, it can be useful, for example, when considering the distribution of the process $Y_t^\e$ in Theorem~\ref{thm:linear_dynamics} starting from different points on the trajectory of $y_t$. Such result is used in \cite{Hamiltonian} in order to establish the regularity of the density of fast-slow systems of the form \eqref{eq:fast-slow} starting from different points on $z_t$.
    Some of the steps in the proof of regularity in \cite{Hamiltonian} involve a linearization of \eqref{eq:fast-slow}, which yields a process of the form \eqref{eq:linear_dynamics}, whose distribution can be described using Theorem~\ref{thm:linear_dynamics}.
\end{remark}

\section{Preliminaries}
\label{sec:prelim}
Throughout the rest of the paper, we assume that the conditions \hyperlink{as:markov_process}{(A1)-(A4)} are satisfied.
In order to understand the limiting distribution of $S_T$, we study the limiting behavior of the characteristic function, which can be represented as a product of the operators $Q(t,\alpha,\beta)$. Since $T$ is not necessarily an integer, we define a special operator $\tilde Q(t,T)$ that will serve as the last factor in the product. Namely, for each $f\in\mathcal B$, we define
\begin{equation}
\label{def:start_end_operator}
    \Tilde Q(t,T)f(x)=\E_x\left[\exp\left(it\cdot\int_0^{T-\lt}b\left(\frac{\lt+s}{T},X_s\right)ds \right)f\left(X_{T-\lt}\right)\right].
\end{equation}
From the definition \eqref{eq:def_of_S} of $S_T$, we immediately get the following lemma.
\begin{lemma}
\label{lem:Nagaev}
     For each $f\in\mathcal B$ and each probability measure $\mu$,
    \begin{equation}
    \label{eq:nagaev}
    \E_\mu[e^{it\cdot S_T}f(X_{T})]=\left\langle\mu,\prod_{k=0}^{\lt-1}Q\left(t,\frac{k}{T},\frac{k+1}{T}\right)\Tilde Q(t,T)f\right\rangle,
    \end{equation}
    where the product of the operators $\prod_{k=1}^n A_k$ means $A_1A_2...A_n$. 
\end{lemma}
The operator in \eqref{def:start_end_operator} does not play a role in the limiting distribution and the introduction of it is purely for technical reasons. We will not touch it until in the proof of the main result.
It will be proved, using Assumption \hyperlink{as:non-arithmetic}{(A4)}, that the value of \eqref{eq:nagaev} decays exponentially fast as $T\to\infty$ if $t$ is away from the origin.
To deal with the situation where $t$ is close to the origin, we need to study the behavior of $Q(t,\alpha,\beta)$ near the origin.
In fact, since $Q(t,\alpha,\beta)$ can be treated as a perturbation of $Q=Q(0,\alpha,\beta)$, the decomposition in the form of \eqref{eq:multi-decomposition} and the quasi-compactness can be extended locally to all $t$ close to the origin. 
This result is given explicitly in the following perturbation theorem.
The theorem is similar to Theorem 3.8 in \cite{HennionHerve}.
The proof is also similar except that now we need to deal with the dependence of the operator $Q(t,\alpha,\beta)$ on the parameters $\alpha$ and $\beta$ due to the time-inhomogeneity of the function $b$, and we need to establish additional differentiability of various terms in the decomposition and prove the existence of a neighborhood $I$ that would work for all the values of the parameters.
We include the proof to address these points.

\begin{theorem}[Perturbation Theorem]
\label{thm:multi-perturbation}
    There exist a neighborhood $I$ of the origin, $C>0$, and $r\in(0,1)$ such that, for $t\in I$, $\alpha,\beta\in[0,1]$, $Q(t,\alpha,\beta)$ has the following decomposition:
    \begin{equation}
        \label{eq:decomposition}
        Q(t,\alpha,\beta)=\lambda(t,\alpha,\beta)v(t,\alpha,\beta)\otimes\varphi(t,\alpha,\beta)+N(t,\alpha,\beta),
    \end{equation} 
    where 
    \begin{enumerate}[(i)]
        \item $\lambda(t,\alpha,\beta)\in\mathcal C^2(I\times[0,1]^2,\mathbb C)$;
        \item $v(t,\alpha,\beta)\in\mathcal C^2(I\times[0,1]^2,\mathcal B)$ and $\|v(t,\alpha,\beta)\|=1$;
        \item $\varphi(t,\alpha,\beta)\in\mathcal C^2(I\times[0,1]^2,\mathcal B')$ and $\langle \varphi(t,\alpha,\beta),v(t,\alpha,\beta)\rangle =1$;
        \item $\mathcal B=F(t,\alpha,\beta)\oplus H(t,\alpha,\beta)$, $F(t,\alpha,\beta)=\mathrm {span}\{v(t,\alpha,\beta)\}$, and $H(t,\alpha,\beta)=\{h:\langle \varphi(t,\alpha,\beta),h\rangle =0\}$;
        \item $N(t,\alpha,\beta)=Q(t,\alpha,\beta)\Pi_{H(t,\alpha,\beta)}\in\mathcal C^2(I\times[0,1]^2,\mathcal L_{\mathcal B})$ and $r(N(t,\alpha,\beta))/|\lambda(t,\alpha,\beta)|<r$ for all $t\in I$ and $\alpha,\beta\in[0,1]$;
        \item $\|\nabla_t^2(N(t,\alpha,\beta)^n)\|<C$ for all $\alpha,\beta\in[0,1]$ and $n\in\mathbb N$.
    \end{enumerate}
    
\end{theorem}
\begin{proof}
    (1) Let us temporarily fix $\alpha_0,\beta_0\in[0,1]$. Recall that $H=\{f\in\mathcal B:\langle\nu,f\rangle=0\}$. Note that for $h\in H$, $\alpha_0,\beta_0\in[0,1]$,
        \begin{equation*}
            \langle\nu,Q(0,\alpha_0,\beta_0)(1+h)\rangle=\langle\nu,Q(1+h)\rangle=\langle Q^*\nu,1+h\rangle=\langle\nu,1+h\rangle=1.
        \end{equation*}
        There exists a neighborhood $J_1(\alpha_0,\beta_0)$ in $\mathbb R^d\times[0,1]^2$ of $(0,\alpha_0,\beta_0)$ such that for all $(t,\alpha,\beta)\in J_1(\alpha_0,\beta_0)$ and $h\in H$ of norm no more than $1$, we have $\langle\nu,Q(t,\alpha,\beta)(1+h)\rangle\not=0$. Then we can define a function $F:J_1(\alpha_0,\beta_0)\times \{h\in H:\|h\|\leq1\}\to H$ by
        \begin{equation*}
            F(t,\alpha,\beta,h)=\frac{Q(t,\alpha,\beta)(1+h)}{\langle\nu,Q(t,\alpha,\beta)(1+h)\rangle}-(1+h).
        \end{equation*}
        Note that $F(0,\alpha_0,\beta_0,0)=0$. By the fact that
        \begin{equation*}
            F(0,\alpha_0,\beta_0,h)=\frac{Q(1+h)}{\langle\nu,Q(1+h)\rangle}-(1+h)=Qh-h,
        \end{equation*}
        we see that the partial differential of $F$ at the point $(0,\alpha_0,\beta_0,0)$ w.r.t. $h$ is $Q|_{H}-1$ which is an isomorphism of $H$ since $1$ is not in the spectrum of $Q|_{H}$ due to \eqref{eq:multi-decomposition}. 
        Hence, by the implicit function theorem, 
        there exists a neighborhood $J_2(\alpha_0,\beta_0)\subset J_1(\alpha_0,\beta_0)$ of $(0,\alpha_0,\beta_0)$, a neighborhood $H_0\subset H$ of $0$, and a unique function $h(t,\alpha,\beta)\in\mathcal C^2(J_2(\alpha_0,\beta_0),H_0)$ such that for all $(t,\alpha,\beta)\in J_2(\alpha_0,\beta_0)$, $F(t,\alpha,\beta,h(t,\alpha,\beta))=0$. 
        And we define \begin{align*}
            v(t,\alpha,\beta)&=\frac{1+h(t,\alpha,\beta)}{\|1+h(t,\alpha,\beta)\|}\in\mathcal C^2(J_2(\alpha_0,\beta_0),\mathcal B)\\  \lambda(t,\alpha,\beta)&=\langle\nu,Q(t,\alpha,\beta)(1+h(t,\alpha,\beta))\rangle\in\mathcal C^2(J_2(\alpha_0,\beta_0),\mathbb C).
        \end{align*}
        Note that $h(0,\alpha,\beta)=0$, $v(0,\alpha,\beta)=1$, and $\lambda(0,\alpha,\beta)=1$ on $J_2(\alpha_0,\beta_0)$.

    (2) For $f\in\mathcal B$, let $f^\perp=\{\psi\in\mathcal B':\langle\psi,f\rangle=0\}$.
    Let $H'=1^{\perp}=\{\psi\in\mathcal B':\langle\psi,1\rangle=0\}$. Define the function $G_1: J_2(\alpha_0,\beta_0)\times H'\to \mathcal B'$ by
    \begin{equation*}
        G_1(t,\alpha,\beta,\psi)=Q(t,\alpha,\beta)^*(\nu+\psi)-\lambda(t,\alpha,\beta)(\nu+\psi).
    \end{equation*}
    Note that $\mathcal B'=H'\oplus \mathrm{span}\{\nu\}$. Define the projection $\pi$ onto $H'$ by $\pi\psi=\psi-\langle\psi,1\rangle\nu$ for each $\psi\in\mathcal B'$,
    and the function $G:J_2(\alpha_0,\beta_0)\times H'\to H'$ by the composition $G=\pi\circ G_1$. 
    By the continuity of $v(t,\alpha,\beta)$, there exists a neighborhood $J_3(\alpha_0,\beta_0)$ smaller than $J_2(\alpha_0,\beta_0)$ such that $\langle\nu,v(t,\alpha,\beta)\rangle\not=0$ for all $(t,\alpha,\beta)\in J_3(\alpha_0,\beta_0)$. So, for each $\psi\in\mathcal B'$,
    \begin{equation*}
        \psi=\frac{\langle\psi,v(t,\alpha,\beta)\rangle}{\langle\nu,v(t,\alpha,\beta)\rangle}\nu+\psi-\frac{\langle\psi,v(t,\alpha,\beta)\rangle}{\langle\nu,v(t,\alpha,\beta)\rangle}\nu,
    \end{equation*}
    and $\mathcal B'=v(t,\alpha,\beta)^\perp\oplus \mathrm{span}\{\nu\}$. Therefore, for each $(t,\alpha,\beta)\in J_3(\alpha_0,\beta_0)$, $\pi$ is bijective from $v(t,\alpha,\beta)^\perp$ to $H'$.
    On the other hand, since $G(0,\alpha_0,\beta_0,0)=0$ and the partial differential of $G$ at point $(0,\alpha_0,\beta_0,0)$ w.r.t. $\psi$ is $Q^*|_{H'}-1$, which is an isomorphism of $H'$, we can find a neighborhood $J_4(\alpha_0,\beta_0)\subset J_3(\alpha_0,\beta_0)$ of $(0,\alpha_0,\beta_0)$, a neighborhood $H_0'\subset H'$ of $0$, and a unique function $\psi\in\mathcal C^2(J_4(\alpha_0,\beta_0), H_0')$ such that, for all $(t,\alpha,\beta)\in J_4(\alpha_0,\beta_0)$, $G(t,\alpha,\beta,\psi(t,\alpha,\beta))=0$ and $\langle\nu+\psi(t,\alpha,\beta),v(t,\alpha,\beta)\rangle\not=0$. 
    By the fact that $\pi$ is bijective from $v(t,\alpha,\beta)^\perp$ to $H'$, $G_1(t,\alpha,\beta,\psi(t,\alpha,\beta))=0$. 
    Then it is clear that $\nu+\psi(t,\alpha,\beta)$ is an eigenvector of $Q^*(t,\alpha,\beta)$. 
    We define \begin{equation*}
        \varphi(t,\alpha,\beta)=\frac{\nu+\psi(t,\alpha,\beta)}{\langle\nu+\psi(t,\alpha,\beta),v(t,\alpha,\beta)\rangle}\in\mathcal C^2(J_4(\alpha_0,\beta_0),\mathcal B'),
    \end{equation*} and define $N(t,\alpha,\beta)\in C^2(J_4(\alpha_0,\beta_0),\mathcal L_{\mathcal B})$ using \eqref{eq:decomposition}. 
    Again, it is clear that $\psi(0,\alpha,\beta)=0$, $\varphi(0,\alpha,\beta)=\nu$, and $N(0,\alpha,\beta)=N$ on $J_4(\alpha_0,\beta_0)$.

    (3) For each $f\in\mathcal B$, we have the equality \begin{equation*}
        f=\langle\varphi(t,\alpha,\beta),f\rangle v(t,\alpha,\beta)+(f-\langle\varphi(t,\alpha,\beta),f\rangle v(t,\alpha,\beta)).
    \end{equation*} Hence, we have the decomposition of $\mathcal B$ as claimed, and it is not hard to verify that $N(t,\alpha,\beta)=Q(t,\alpha,\beta)\Pi_{H(t,\alpha,\beta)}$. Observe that $r(N)<1$ (by the quasi-compactness of $Q$), $\lambda(t,\alpha,\beta)$ and $N(t,\alpha,\beta)$ are continuous, and spectral radius is an upper-semicontinuous function of the operator. Therefore, $r(N(t,\alpha,\beta))/|\lambda(t,\alpha,\beta)|<r:=(r(N)+1)/2<1$ holds on a neighborhood $J_5(\alpha_0,\beta_0)\subset J_4(\alpha_0,\beta_0)$.

    Finally, to get the bound in (\romannumeral6), we express $\nabla_t^2(N(t,\alpha,\beta)^n)$ as a sum of $n^2 $ terms of the forms
    \begin{equation}
    \label{eq:T_form1}
        V=N(t,\alpha,\beta)^{n_1}\nabla_t N(t,\alpha,\beta)N(t,\alpha,\beta)^{n_2}\nabla_t N(t,\alpha,\beta)^*N(t,\alpha,\beta)^{n_3}
    \end{equation} with $n_1+n_2+n_3=n-2$, $n_1,n_2,n_3\geq0$,
    or
    \begin{equation}
    \label{eq:T_form2}
        V=N(t,\alpha,\beta)^{n_1}\nabla_t^2 N(t,\alpha,\beta)N(t,\alpha,\beta)^{n_2}
    \end{equation}with $n_1+n_2=n-1$, $n_1,n_2\geq0$. 
    Let us choose $m_0\in\mathbb N$ such that $\|N^{m_0}\|<(\frac{2r(N)+1}{3})^{m_0}$. 
    Since $N(t,\alpha,\beta)\in\mathcal C^2(I\times[0,1]^2,\mathcal L_{\mathcal B})$, there exist a neighborhood $J(\alpha_0,\beta_0)\subset J_5(\alpha_0,\beta_0)$ and a constant $C_1(\alpha_0,\beta_0)\geq1$ such that $\|N(t,\alpha,\beta)^{m_0}\|<r^{m_0}$, and $N(t,\alpha,\beta)$, $\nabla_tN(t,\alpha,\beta)$, and $\nabla_t^2N(t,\alpha,\beta)$ are bounded by $C_1(\alpha_0,\beta_0)$ on $J(\alpha_0,\beta_0)$. 
    Let $C_2(\alpha_0,\beta_0)=(C_1(\alpha_0,\beta_0)/r)^{m_0}$. 
    Then, for every $m\in\mathbb N$, with $m=km_0+j$, $0\leq j<m_0$, we have on $J(\alpha_0,\beta_0)$ that
    \begin{equation*}
        \|N(t,\alpha,\beta)^m\|\leq \|N(t,\alpha,\beta)^{m_0}\|^k\cdot \|N(t,\alpha,\beta)\|^j\leq r^m\cdot\frac{C_1(\alpha_0,\beta_0)^j}{r^j}\leq C_2(\alpha_0,\beta_0) r^m.
    \end{equation*}
    Hence, we obtain that, for each term $V$ in the form of \eqref{eq:T_form1} or \eqref{eq:T_form2}, $\|V||\leq C_2(\alpha_0,\beta_0)^5\cdot r^{n-2}$. Since there are $n^2$ terms in total, we have that 
    \begin{equation*}
        \|\nabla_t^2(N(t,\alpha,\beta)^n)\|\leq C_2(\alpha_0,\beta_0)^5n^2\cdot r^{n-2}\to0,
    \end{equation*}
    as $n\to\infty$. Therefore, the left-hand side is bounded uniformly in $n$ and $(t,\alpha,\beta)\in J(\alpha_0,\beta_0)$.
    
    (4) So far, we showed that $\lambda(t,\alpha,\beta),v(t,\alpha,\beta)$, $\varphi(t,\alpha,\beta)$, and $N(t,\alpha,\beta)$ that satisfy the decomposition \eqref{eq:decomposition} and the conditions (\romannumeral1)-(\romannumeral6) stated in the theorem exist  on $\bigcup_{\alpha_0,\beta_0\in[0,1]}J(\alpha_0,\beta_0)$. 
    Then, by the compactness of $[0,1]^2$, we can find a neighborhood $I$ of the origin in $\mathbb R^d$ such that the functions above exist on $I\times[0,1]^2$ and the conditions (\romannumeral5) and (\romannumeral6) are satisfied with uniform constants.
\end{proof}

\section{Product of the operators}
\label{sec:product}
As suggested by Lemma~\ref{lem:Nagaev}, the main ingredient of the proof is to study the limiting behavior of the product of the operators. In this section, we deal with two situations where $t$ is close to or away from the origin. The product contains the information about the characteristics of the limiting distribution in the first situation and vanishes as $T$ tends to infinity in the second situation. Our arguments involve different products. To avoid confusion, we use the convention: for $p_k\in\mathbb C$, $\prod_{k=j}^{j-1}p_j=1$; and, for $P_k\in\mathcal L_{\mathcal B}$, $\prod_{k=j}^{j-1}P_j=\mathrm{id}$ is the identity operator.

\subsection{For \texorpdfstring{$t$}{t} near \texorpdfstring{$0$}{0}}
In this section, we apply the perturbation theorem to deal with the product where $t$ is close to the origin. 
To present cleaner arguments, we assume that the operator $N$ in \eqref{eq:multi-decomposition} has its norm less than $1$. 
This assumption is not restrictive since, by the uniform ergodicity assumption \hyperlink{as:mixing}{(A2)}, if we define $Q=P_n$ with $n$ large enough instead of $Q=P_1$, then the total variation between the distribution of $X_n$ and the invariant measure $\nu$ will be small enough, and the norm of $N$ will be less than $1$. 
The perturbation theorem still holds with the relation in $\textit{(\romannumeral5)}$ replaced by 
\begin{equation}
\label{eq:additional_assumption}
    \|N(t,\alpha,\beta)\|/|\lambda(t,\alpha,\beta)|<r,
\end{equation} and the left-hand side of \eqref{eq:nagaev} can still be represented as a product of those perturbed operators.
\begin{proposition}
\label{prop:multi-operator_product}
Let $I$ be defined as in Theorem~\ref{thm:multi-perturbation}. There exist constants $M$ and $N$ large enough such that, for all $T>N$ and all $t\in I$, we have the formula for the product of the operators, for $f\in\mathcal B$ and $0\leq i\leq\lt-1$,
    \begin{align}
        &\prod_{k=i}^{\lt-1} Q\left(t,\frac{k}{T},\frac{k+1}{T}\right)f\nonumber\\
        &=\prod_{k=i}^{\lt-1}\lambda\left(t,\frac{k}{T},\frac{k+1}{T}\right)\prod_{k=i}^{\lt-2}\left\langle \varphi\left(t,\frac{k}{T},\frac{k+1}{T}\right),v\left(t,\frac{k+1}{T},\frac{k+2}{T}\right)\right\rangle \nonumber\\
        &\quad\quad\quad\quad\cdot\left\langle \varphi\left(t,\frac{\lt-1}{T},\frac{\lt}{T}\right),f\right\rangle v\left(t,\frac{i}{T},\frac{i+1}{T}\right)\nonumber\\
        &+\prod_{k=i}^{\lt-1}\lambda\left(t,\frac{k}{T},\frac{k+1}{T}\right)\|f\|\left(p_{i,T}^t\cdot  v\left(t,\frac{i}{T},\frac{i+1}{T}\right)+q_{i,T}^t\cdot h\left(t,\frac{i}{T},\frac{i+1}{T}\right)\right), \label{eq:multi-operator_product}  
    \end{align}
    where $h(t,\alpha,\beta)\in H(t,\alpha,\beta)$, $\|h(t,\alpha,\beta)\|=1$, $|p_{i,T}^t|<M/T$, and $|q_{i,T}^t|<r^{\lt-i}+M/2T$. In particular, $|p_{0,T}^t|,|q_{0,T}^t|<M/T$. The bounds on $p_{i,T}^t,q_{i,T}^t$ hold uniformly in $f\in\mathcal B$.
\end{proposition}
The proof of this result requires careful treatment of error terms induced by the shifts of the eigenspace corresponding to the top eigenvalue. Roughly speaking, the idea is to show that, compared with the first term in \eqref{eq:multi-operator_product}, which is obtained by exclusively looking at the vectors and projections on the eigenspace, error terms that emerge because of the shifts of the eigenspace and the presence of the operator $N(t,\frac{k}{T},\frac{k+1}{T})$ are negligible. 

\begin{proof}
    We apply Theorem~\ref{thm:multi-perturbation} to control the error brought by the shift of the eigenspace. Namely, there exists $K>1$ such that the following relations hold for $t\in I$:
    \begin{enumerate}[(i)]
        \item For all $1\leq k\leq \lt-1$,
        \begin{equation}\begin{aligned}
        \label{eq:varphi_v}
            &\left|\left\langle\varphi\left(t,\frac{k-1}{T},\frac{k}{T}\right),v\left(t,\frac{k}{T},\frac{k+1}{T}\right)\right\rangle-1\right|\\
            &= \left|\left\langle\varphi\left(t,\frac{k-1}{T},\frac{k}{T}\right),v\left(t,\frac{k}{T},\frac{k+1}{T}\right)-v\left(t,\frac{k-1}{T},\frac{k}{T}\right)\right\rangle\right|\\
            &\leq K/T.
        \end{aligned}\end{equation} Consequently,
        \begin{equation}
        \label{eq:varphi_v_product}
            \left|\prod_{k=i}^{\lt-1}\left\langle\varphi\left(t,\frac{k-1}{T},\frac{k}{T}\right),v\left(t,\frac{k}{T},\frac{k+1}{T}\right)\right\rangle\right|\leq e^K.
        \end{equation}
        Moreover, by \eqref{eq:additional_assumption},
        \begin{equation}
        \label{eq:N_v}
            \begin{aligned}
                &\left\|N\left(t,\frac{k-1}{T},\frac{k}{T}\right)v\left(t,\frac{k}{T},\frac{k+1}{T}\right)\right\|\\
                &=\left\|N\left(t,\frac{k-1}{T},\frac{k}{T}\right)\Pi_{H(t,\frac{k-1}{T},\frac{k}{T})}v\left(t,\frac{k}{T},\frac{k+1}{T}\right)\right\|\\
            &\leq\left\|N\left(t,\frac{k-1}{T},\frac{k}{T}\right)\right\|\left\|\Pi_{H(t,\frac{k-1}{T},\frac{k}{T})}v\left(t,\frac{k}{T},\frac{k+1}{T}\right)\right\|\\
            &\leq r\left|\lambda\left(t,\frac{k-1}{T},\frac{k}{T}\right)\right|\left\|\Pi_{H(t,\frac{k-1}{T},\frac{k}{T})}\left(v\left(t,\frac{k}{T},\frac{k+1}{T}\right)-v\left(t,\frac{k-1}{T},\frac{k}{T}\right)\right)\right\|\\
            &\leq r\left|\lambda\left(t,\frac{k-1}{T},\frac{k}{T}\right)\right|\cdot\frac{K}{T}.
            \end{aligned}
        \end{equation}
        \item For all $1\leq k\leq \lt-1$, $h\in H(t,\frac{k}{T},\frac{k+1}{T})$, $\|h\|=1$, 
        \begin{align}
            \left|\left\langle\varphi\left(t,\frac{k-1}{T},\frac{k}{T}\right),h\right\rangle\right|&= \left|\left\langle\varphi\left(t,\frac{k-1}{T},\frac{k}{T}\right)-\varphi\left(t,\frac{k}{T},\frac{k+1}{T}\right),h\right\rangle\right|\leq \frac{K}{T},\label{eq:varphi_h}\\
            \left\|N\left(t,\frac{k-1}{T},\frac{k}{T}\right)h\right\|&\leq r\left|\lambda\left(t,\frac{k-1}{T},\frac{k}{T}\right)\right|.\label{eq:N_h}
        \end{align}
    \end{enumerate}
    Furthermore, we can choose $K$ sufficiently large so that $\|\varphi(t,\alpha,\beta)\|$ is bounded by $K$ for all $t\in I$, $\alpha,\beta\in[0,1]$.
    Let $a$ denote $2K^2e^K/T$.
    We prove by backwards induction that, for all $t\in I$, $0\leq i\leq\lfloor T\rfloor-1$, and all $T$ sufficiently large, \eqref{eq:multi-operator_product} holds with
    \begin{equation}
    \label{eq:induction}
        \begin{aligned}
            |p_{i,T}|&\leq a\cdot\left(\frac{r}{1-r}[(1+a)^{\lt-i}-1]+\frac{(1+a)^{\lt-i}-r^{\lt-i}}{1+a-r}\right),\\
            |q_{i,T}|&\leq r^{\lt-i}+a\cdot\frac{r}{1-r}.
        \end{aligned}
    \end{equation}
    In particular, it is not hard to see that this implies the desired result.
    
    Let us verify that \eqref{eq:induction} holds if $i=\lt-1$. By \eqref{eq:decomposition},
    \begin{equation*}
        \begin{aligned}
            Q\left(t,\frac{\lt-1}{T},\frac{\lt}{T}\right)f&=\lambda\left(t,\frac{\lt-1}{T},\frac{\lt}{T}\right)\left\langle\varphi\left(t,\frac{\lt-1}{T},\frac{\lt}{T}\right),f\right\rangle v\left(t,\frac{\lt-1}{T},\frac{\lt}{T}\right)\\
            &\quad+N\left(t,\frac{\lt-1}{T},\frac{\lt}{T}\right)f.
        \end{aligned}
    \end{equation*}
    We have $p_{\lt-1,T}=0$ and, by \eqref{eq:additional_assumption},
    \begin{equation*}
        \left|q_{\lt-1,T}\right|\leq \left\|N\left(t,\frac{\lt-1}{T},\frac{\lt}{T}\right)\right\|/\lambda\left(t,\frac{\lt-1}{T},\frac{\lt}{T}\right)\leq r.
    \end{equation*}
    By assuming the validity of \eqref{eq:induction} at $i$, we obtain 
    \begin{align}
        &\prod_{k=i-1}^{\lt-1}Q\left(t,\frac{k}{T},\frac{k+1}{T}\right)f\nonumber\\
        &=Q\left(t,\frac{i-1}{T},\frac{i}{T}\right)\prod_{k=i}^{\lt-1}Q\left(t,\frac{k}{T},\frac{k+1}{T}\right)f\nonumber\\
        &=\prod_{k=i-1}^{\lt-1}\lambda\left(t,\frac{k}{T},\frac{k+1}{T}\right)\prod_{k=i-1}^{\lt-2}\left\langle \varphi\left(t,\frac{k}{T},\frac{k+1}{T}\right),v\left(t,\frac{k+1}{T},\frac{k+2}{T}\right)\right\rangle\nonumber\\
        &\quad\quad\quad\quad\cdot\left\langle\varphi\left(t,\frac{\lt-1}{T},\frac{\lt}{T}\right),f\right\rangle v\left(t,\frac{i-1}{T},\frac{i}{T}\right)\nonumber\\
        &\quad+\prod_{k=i-1}^{\lt-1}\lambda\left(t,\frac{k}{T},\frac{k+1}{T}\right)\|f\|\left\langle\varphi\left(t,\frac{i-1}{T},\frac{i}{T}\right),v\left(t,\frac{i}{T},\frac{i+1}{T}\right)\right\rangle p_{i,T}v\left(t,\frac{i-1}{T},\frac{i}{T}\right)\label{term:p_k1}\\
        &\quad+\prod_{k=i-1}^{\lt-1}\lambda\left(t,\frac{k}{T},\frac{k+1}{T}\right)\|f\|\left\langle\varphi\left(t,\frac{i-1}{T},\frac{i}{T}\right),h\left(t,\frac{i}{T},\frac{i+1}{T}\right)\right\rangle q_{i,T}v\left(t,\frac{i-1}{T},\frac{i}{T}\right)\label{term:p_k2}\\
        &\quad+\prod_{k=i-1}^{\lt-1}\lambda\left(t,\frac{k}{T},\frac{k+1}{T}\right)\prod_{k=i}^{\lt-2}\left\langle \varphi\left(t,\frac{k}{T},\frac{k+1}{T}\right),v\left(t,\frac{k+1}{T},\frac{k+2}{T}\right)\right\rangle\left\langle\varphi\left(t,\frac{\lt-1}{T},\frac{\lt}{T}\right),f\right\rangle\nonumber\\
        &\quad\quad\quad\quad\cdot N\left(t,\frac{i-1}{T},\frac{i}{T}\right)v\left(t,\frac{i}{T},\frac{i+1}{T}\right)/\lambda\left(t,\frac{i-1}{T},\frac{i}{T}\right)\label{term:q_k1}\\
        &\quad+\prod_{k=i-1}^{\lt-1}\lambda\left(t,\frac{k}{T},\frac{k+1}{T}\right)\|f\|p_{i,T}N\left(t,\frac{i-1}{T},\frac{i}{T}\right)v\left(t,\frac{i}{T},\frac{i+1}{T}\right)/\lambda\left(t,\frac{i-1}{T},\frac{i}{T}\right)\label{term:q_k2}\\
        &\quad+\prod_{k=i-1}^{\lt-1}\lambda\left(t,\frac{k}{T},\frac{k+1}{T}\right)\|f\|q_{i,T}N\left(t,\frac{i-1}{T},\frac{i}{T}\right)h\left(t,\frac{i}{T},\frac{i+1}{T}\right)/\lambda\left(t,\frac{i-1}{T},\frac{i}{T}\right).\label{term:q_k3}
        \end{align}
    We deduce that, by \eqref{eq:varphi_v} and \eqref{eq:varphi_h},
    \begin{equation}
    \label{eq:induction_p}
        |p_{i-1,T}|\leq (1+a)|p_{i,T}|+a|q_{i,T}|,
    \end{equation}where the two terms on the right-hand side correspond to \eqref{term:p_k1} and \eqref{term:p_k2}, respectively; and, by \eqref{eq:varphi_v_product}, \eqref{eq:N_v}, and \eqref{eq:N_h},
    \begin{equation}
    \label{eq:induction_q}
        |q_{i-1,T}|\leq \frac{1}{2}ar+\frac{1}{2}ar|p_{i,T}|+r|q_{i,T}|,
    \end{equation}where the three terms on the right-hand side correspond to \eqref{term:q_k1}, \eqref{term:q_k2}, and \eqref{term:q_k3}, respectively.
    
    Now using \eqref{eq:induction_p} and \eqref{eq:induction_q} and backwards induction on $i$, it is easy to see that \eqref{eq:induction} holds.
    Indeed, at each step of induction, we first show that $|p_{i,T}^t|<1$. So $\eqref{eq:induction_q}$ implies that 
    \begin{equation}
    \label{eq:improved_induction_q}
        |q_{i-1,T}|\leq ar+r|q_{i,T}|,
    \end{equation}
    when assuming \eqref{eq:induction} for $p_{i,T}^t$ and $q_{i,T}^t$.
    Then \eqref{eq:induction} for $i-1$ instead of $i$ can be established using \eqref{eq:induction_p} and \eqref{eq:improved_induction_q}. As we noted, \eqref{eq:induction} implies the statement in the lemma.
\end{proof}

\subsection{For \texorpdfstring{$t$}{t} away from \texorpdfstring{$0$}{0}}
As noted earlier, we expect that the product decays exponentially fast for all $t$ that are at a uniformly positive distance from the origin, so it does not contribute to \eqref{eq:nagaev}.
\begin{lemma}
\label{lem:multi-mstepnorm}
Suppose that $P(t,\alpha,\beta)\in\mathcal C(K\times[0,1]^2,\mathcal L_{\mathcal B})$ where $K$ is a compact set in $\mathbb R^d$ and that $r(P(t,\alpha,\beta))<r_0$ for all $t\in K$, $\alpha,\beta\in[0,1]$. Then
\begin{enumerate}[(i)]
    \item There exists $r_1<r_0$ and $m_0$ such that $\|P(t,\alpha,\beta)^m\|<r_1^m$ for all $t\in K$, $\alpha,\beta\in[0,1]$, and $m\geq m_0$.
    \item There exist $r_2<r_0$ and $m_0$ such that for each $m\geq m_0$ there exists $\delta>0$ such that for all $t\in K$, $$\left\|\prod_{k=1}^m P(t,\alpha_k,\beta_k)\right\|<r_2^m$$ if $|\alpha_k-\alpha_1|<\delta$ and $|\beta_k-\beta_1|<\delta$ for all $1\leq k\leq m$.  
\end{enumerate}
\end{lemma}
\begin{proof}
    \begin{enumerate}[(i)]
        \item This result follows from the upper semi-continuity of the spectral radius, as in the proof of Corollary III.13 in \cite{HennionHerve}.
        \item Let $m_0$ be defined as in (\romannumeral1) and fix arbitrary $r_2\in(r_1,r)$ and $m\geq m_0$. 
        Define
        \[P_m(t,\alpha_1,...,\alpha_m,\beta_1,...,\beta_m)=\prod_{k=1}^m P(t,\alpha_k,\beta_k).\]
        Since $P(t,\alpha,\beta)$ is continuous in $t$, $\alpha$, and $\beta$, $P_m$ is continuous on $K\times[0,1]^{2m}$, hence is uniformly continuous. Then there exists $\delta>0$, depending on $m$, such that
        \[\|P_m(t,\alpha_1,...,\alpha_1,\beta_1,...,\beta_1)-P_m(t,\alpha_1,...,\alpha_m,\beta_1,...,\beta_m)\|\leq r_2^m-r_1^m,\]
        if $|\alpha_k-\alpha_1|<\delta$ and $|\beta_k-\beta_1|<\delta$ for all $1\leq k\leq m$. Therefore, the result follows.\qedhere
    \end{enumerate}
\end{proof}

\begin{lemma}
\label{lem:exponential_convergence_away_from_0}
    For every compact set $K\subset\mathbb R^d\setminus\{0\}$, there exist constants $r_K\in(0,1)$ and $N$ large enough such that, for all $T>N$ and all $t\in K$,
    \begin{equation*}
        \left\|\prod_{k=0}^{\lfloor T\rfloor-1}Q\left(t,\frac{k}{T},\frac{k+1}{T}\right)\right\|<r_K^{\lfloor T\rfloor}.
    \end{equation*}
\end{lemma}
\noindent\textit{Proof.}
    By Lemma~\ref{lem:multi-mstepnorm}.(\romannumeral2) and Assumption \hyperlink{as:non-arithmetic}{(A4)}, there exist $r_2<1$, $m\in\mathbb N$ large enough, and $\delta>0$ such that
    $$\left\|\prod_{k=1}^m Q(t,\alpha_k,\beta_k)\right\|<r_2^m$$
    holds if $|\alpha_k-\alpha_1|<\delta$ and $|\beta_k-\beta_1|<\delta$ for all $1\leq k\leq m$. 
    Now fix arbitrary $r_K\in(r_2,1)$. 
    Choose $l_0\in\mathbb N$ such that $(r_2/r_K)^{l_0}<r_K$ and choose $N>3m(l_0\vee1/\delta)$. Let $\lt=\tilde lm+j$ with $0\leq j<m$. Then we have $\tilde l\geq l_0$ and 
    \begin{align*}
        \left\|\prod_{k=0}^{\lt-1} Q\left(t,\frac{k}{T},\frac{k+1}{T}\right)\right\|&\leq\left\|\prod_{l=1}^{\tilde l}\prod_{k=(l-1)m}^{lm-1} Q\left(t,\frac{k}{T},\frac{k+1}{T}\right)\right\|\cdot\left\|\prod_{k=\Tilde{l}m}^{\lt-1}Q\left(t,\frac{k}{T},\frac{k+1}{T}\right)\right\|\\
        &\leq\prod_{l=1}^{\tilde l}\left\|\prod_{k=(l-1)m}^{lm-1} Q\left(t,\frac{k}{T},\frac{k+1}{T}\right)\right\|\leq r_2^{\tilde lm}<r_K^{\tilde lm}\cdot r_K^m<r_K^{\lt}.\tag*{\qed}
    \end{align*}

\subsection{Product of the top eigenvalues}
In this section, we obtain the asymptotic value of the product of the top eigenvalues by the Taylor expansion near $0$. The first result concerns the first and second derivatives of $\lambda(t,\alpha,\beta)$ w.r.t. $t$ at $0$.
\begin{lemma}
\label{lem:asymptotics}
\begin{enumerate}[(i)]
    \item For all $\alpha,\beta\in[0,1]$, we have $\nabla_t\lambda(0,\alpha,\beta)=0$.
    \item For all $\alpha\in[0,1]$, $\nabla_t^2\lambda(0,\alpha,\alpha)$ is real and negative-definite. Moreover,
    \begin{equation}
    \label{eq:second_derivative}
        \nabla_t^2\lambda(0,\alpha,\alpha)=-\lim_{T\to\infty}\frac{1}{T}\E_\nu({S_T^{\alpha}}{S_T^{\alpha}}^*),
    \end{equation}
    where $S_T^{\alpha}=\int_0^T b(\alpha,X_{s})ds$. 
\end{enumerate}
\end{lemma}
\begin{proof}
        \textit{(i).}With the same notation as in Theorem~\ref{thm:multi-perturbation}, we have the decomposition
    \begin{equation*}
        Q(t,\alpha,\beta)=\lambda(t,\alpha,\beta)\Tilde v(t,\alpha,\beta)\otimes\Tilde\varphi(t,\alpha,\beta)+N(t,\alpha,\beta),
    \end{equation*}
    for $t\in I$ and $\alpha,\beta\in[0,1]$,
    where \[\Tilde v(t,\alpha,\beta)=\frac{v(t,\alpha,\beta)}{\langle\nu,v(t,\alpha,\beta)\rangle}\quad\text{and}\quad\Tilde\varphi(t,\alpha,\beta)=\langle\nu,v(t,\alpha,\beta)\rangle\varphi(t,\alpha,\beta).\]
    Compared to the previous decomposition, we no longer have $\|\tilde v(t,\alpha,\beta)\|\equiv1$, but instead have $\langle\Tilde\varphi(0,\alpha,\beta),\Tilde v(t,\alpha,\beta)\rangle\equiv1$, and still have $Q(t,\alpha,\beta)\Tilde v(t,\alpha,\beta)=\lambda(t,\alpha,\beta)\Tilde v(t,\alpha,\beta)$. 
    Differentiate the equations with respect to $t$. Then
    \begin{align}
        &\langle\Tilde\varphi(0,\alpha,\beta),\nabla_t\Tilde v(t,\alpha,\beta)\rangle\equiv0,\nonumber\\
        \nabla_tQ(t,\alpha,\beta)\Tilde v(t,\alpha,\beta)&+Q(t,\alpha,\beta)\nabla_t\Tilde v(t,\alpha,\beta)\nonumber\\
        &=\nabla_t\lambda(t,\alpha,\beta)\Tilde v(t,\alpha,\beta)+\lambda(t,\alpha,\beta)\nabla_t\Tilde v(t,\alpha,\beta).\label{eq:multi-firstderivative}
    \end{align}
    Note that $\tilde v(0,\alpha,\beta)=1$, $\Tilde\varphi(0,\alpha,\beta)=\nu$, $\lambda(0,\alpha,\beta)=1$, and $Q(0,\alpha,\beta)=Q$.
    Applying $\nu$ to \eqref{eq:multi-firstderivative} at $t=0$, we obtain, by Assumption \hyperlink{as:zero_mean}{(A3)},
    \begin{equation}
        \label{eq:first_derivative}
        \nabla_t\lambda(0,\alpha,\beta)=\langle\nu,\nabla_tQ(0,\alpha,\beta)1\rangle=i\E_\nu\left(\int_0^1 b(\alpha+s(\beta-\alpha),X_s)ds\right)=0.
    \end{equation}

    \textit{(ii).}Taking \eqref{eq:first_derivative} into account, for $t=0$ and $\alpha=\beta$, \eqref{eq:multi-firstderivative} reduces to
    \begin{equation}
    \label{eq:for_variance_1}
        (1-Q)\nabla_t\Tilde v(0,\alpha,\alpha)=\nabla_tQ(0,\alpha,\alpha)1=i\E_x\left(\int_0^1 b(\alpha,X_s)ds\right).
    \end{equation}
    Since both $\nabla_t\Tilde v(0,\alpha,\alpha)$ and $i\E_x(\int_0^1 b(\alpha,X_s)ds)$ have zero mean w.r.t $\nu$, they are in the closed subspace $H$. 
    Also recall that $1$ is not in the spectrum of $Q_{|H}$. Therefore, $\nabla_t\Tilde v(0,\alpha,\alpha)$ can be viewed as a solution to \eqref{eq:for_variance_1}, and thus $\nabla_t\Tilde v(0,\alpha,\alpha)$ is purely imaginary, i.e.,
    \begin{equation}
    \label{eq:for_variance_3}
        \nabla_t\Tilde v(0,\alpha,\alpha)=iu(\alpha)
    \end{equation} for a certain real vector-valued bounded function $u(\alpha)$ that depends on the parameter $\alpha$.
    Take second derivatives on both sides of the equality $Q(t,\alpha,\alpha)\Tilde v(t,\alpha,\alpha)=\lambda(t,\alpha,\alpha)\Tilde v(t,\alpha,\alpha)$, apply $\nu$ to the result at $t=0$, and obtain the following equality of matrices (here $*$ stands for transposition without taking the complex conjugate):
    \begin{equation}
    \begin{aligned}
        \label{eq:second_der_lambda}\nabla_t^2\lambda(0,\alpha,\alpha)&=\langle\nu,\nabla_t^2Q(0,\alpha,\alpha)1+\nabla_tQ(0,\alpha,\alpha)\nabla_t\Tilde v(0,\alpha,\alpha)^*\\
        &\quad+\left(\nabla_tQ(0,\alpha,\alpha)\nabla_t\Tilde v(0,\alpha,\alpha)^*\right)^*\rangle.
    \end{aligned}
    \end{equation}
    Let us compute the terms on the right hand side by recalling the definition in \eqref{def:Q(t,alpha,beta)} and \eqref{eq:for_variance_3}:
    \begin{equation}
    \label{eq:der_Q}
    \begin{aligned}
        \langle\nu,\nabla_t^2Q(0,\alpha,\alpha)1\rangle&=-\E_{\nu}\left[\left(\int_0^1 b(\alpha,X_s)ds\right)\left(\int_0^1 b(\alpha,X_s)ds\right)^*\right],\\
        \langle\nu,\nabla_tQ(0,\alpha,\alpha)\nabla_t\Tilde v(0,\alpha,\alpha)^*\rangle&=-\E_\nu\left[\left(\int_0^1 b(\alpha,X_s)ds\right)u(\alpha)(X_1)^*\right].
        \end{aligned}
    \end{equation}
    Also notice that, by \eqref{eq:for_variance_1} and \eqref{eq:for_variance_3},
    \begin{equation}
    \label{eq:revised1}
        \begin{aligned}
            \E_\nu\left[\E_x\left(\int_0^1 b(\alpha,X_s)ds+u(\alpha)(X_1)\right)u(\alpha)(x)^*\right]&=\E_\nu\left[\left((1-Q)u(\alpha)(x)+\E_xu(\alpha)(X_1)\right)u(\alpha)(x)^*\right]\\
            &=\E_\nu [u(\alpha)(x)u(\alpha)(x)^*].
        \end{aligned}
    \end{equation}
    Then, using \eqref{eq:second_der_lambda} and \eqref{eq:der_Q} for the first equality below and  \eqref{eq:revised1} for the second equality below, we obtain that
    \begin{align*}
        &\nabla_t^2\lambda(0,\alpha,\alpha)\\
        &=-\E_\nu\left[\left(\int_0^1 b(\alpha,X_s)ds+u(\alpha)(X_1)\right)\left(\int_0^1 b(\alpha,X_s)ds+u(\alpha)(X_1)\right)^*\right]+\E_\nu[ u(\alpha)(x)u(\alpha)(x)^*]\\
        &=-\E_\nu\left[\left(\int_0^1 b(\alpha,X_s)ds+u(\alpha)(X_1)-u(\alpha)(x)\right)\left(\int_0^1 b(\alpha,X_s)ds+u(\alpha)(X_1)-u(\alpha)(x)\right)^*\right].
    \end{align*}
    Hence $\nabla_t^2\lambda(0,\alpha,\alpha)$ is negative semi-definite. To see that it is negative definite, let us assume that there exists a non-zero vector $t\in\mathbb R^d$ such that $t^*\nabla_t^2\lambda(0,\alpha,\alpha)t=0$. Then we have $\nu$-a.s. that
    \begin{equation*}
        \E_x\left[t\cdot\left(\int_0^1 b(\alpha,X_s)ds+u(\alpha)(X_1)-u(\alpha)(x)\right)\right]^2=0,
    \end{equation*}
    which implies that $t\cdot\left(\int_0^1 b(\alpha,X_s)ds+u(\alpha)(X_1)-u(\alpha)(x)\right)=0$, $\nu$-a.s. However, this contradicts Assumption \hyperlink{as:non-arithmetic}{(A4)}. Indeed, let $f(x)=\exp(it\cdot u(\alpha)(x))$. Then we have 
    \begin{equation*}
        Q(t,\alpha,\alpha)f(x)=\E_x\left[\exp\left(it\cdot\left(\int_0^1 b(\alpha,X_s)ds+u(\alpha)(X_1)-u(\alpha)(x)\right)\right)\right]f(x).
    \end{equation*}
    Then $Q(t,\alpha,\alpha)f(x)=f(x)$, $\nu$-a.s., which contradicts the fact that $r(Q(t,\alpha,\alpha))<1$.

    To prove \eqref{eq:second_derivative}, we write, for $m\in\mathbb N$, 
    \begin{equation*}
        Q(t,\alpha,\alpha)^m=\lambda(t,\alpha,\alpha)^mv(t,\alpha,\alpha)\otimes\varphi(t,\alpha,\alpha)+N^m(t,\alpha,\alpha).
    \end{equation*}
    Taking the second derivatives, we obtain
    \begin{equation*}
    \begin{aligned}
        &\nabla_t^2\left(Q(t,\alpha,\alpha)^m\right)\\
        &=\nabla_t^2(\lambda(t,\alpha,\alpha)^m)v(t,\alpha,\alpha)\otimes\varphi(t,\alpha,\alpha)+\nabla_t(v(t,\alpha,\alpha)\otimes\varphi(t,\alpha,\alpha))\nabla_t(\lambda(t,\alpha,\alpha)^m)^*\\
        &\quad +\nabla_t(\lambda(t,\alpha,\alpha)^m)\nabla_t(v(t,\alpha,\alpha)\otimes\varphi(t,\alpha,\alpha))^*+R_m(t,\alpha,\alpha),
    \end{aligned}
    \end{equation*}
    where $R_m(t,\alpha,\alpha)=\lambda(t,\alpha,\alpha)^m\nabla_t^2(v(t,\alpha,\alpha)\otimes\varphi(t,\alpha,\alpha))+\nabla_t^2(N^m(t,\alpha,\alpha))$. Note that 
    \begin{equation*}
    \begin{aligned}
        \nabla_t^2(\lambda(t,\alpha,\alpha)^m)&=m(m-1)\lambda(t,\alpha,\alpha)^{m-2}\nabla_t\lambda(t,\alpha,\alpha)\nabla_t\lambda(t,\alpha,\alpha)^*\\
        &\quad+m\lambda(t,\alpha,\alpha)^{m-1}\nabla_t^2\lambda(t,\alpha,\alpha).
    \end{aligned}
    \end{equation*}
    Then it follows that $\nabla_t^2(\lambda(t,\alpha,\alpha)^m)(0,\alpha,\alpha)=m\nabla_t^2\lambda(0,\alpha,\alpha)$.
    By Theorem~\ref{thm:multi-perturbation}, $R_m(0,\alpha,\alpha)$ is uniformly bounded for all $m$. Thus, for all $\mu\in\mathcal B_p'$,
    \begin{equation*}
        \sup_{m\geq 1}|\langle\mu,m\nabla_t^2\lambda(0,\alpha,\alpha)\cdot 1-\nabla_t^2(Q^m(0,\alpha,\alpha))\cdot 1\rangle|=\sup_{m\geq1}|\langle \mu,R_m(0,\alpha,\alpha)\cdot 1\rangle|<\infty.
    \end{equation*}
    Therefore, for all initial distribution $\mu\in\mathcal B'_p$,
    \begin{equation*}
        \begin{aligned}
            \nabla_t^2\lambda(0,\alpha,\alpha)&=\lim_{m\to\infty}\frac{1}{m}\langle\mu,\nabla_t^2(Q^m(t,\alpha,\alpha))(0,\alpha,\alpha)\cdot 1\rangle\\
            &=-\lim_{m\to\infty}\frac{1}{m}\E_\mu({S_m^{\alpha}}{S_m^{\alpha}}^*),
        \end{aligned}
    \end{equation*}
    which implies \eqref{eq:second_derivative}.
\end{proof}
We close this section by computing the product of the eigenvalues.
\begin{lemma}
    \label{lem:product_of_eigenvalues}
    For $\tau\in\mathbb R^d$,
    \begin{equation*}\lim_{T\to\infty}\prod_{k=0}^{\lfloor T\rfloor-1}\lambda\left(\frac{\tau}{\sqrt{T}},\frac{k}{T},\frac{k+1}{T}\right)=\exp\left(-\frac{1}{2}\tau^*\Sigma\Sigma^*\tau\right),
    \end{equation*}where $\Sigma\Sigma^*:=-\int_0^1\nabla_t^2\lambda(0,\alpha,\alpha)d\alpha$ is positive-definite.
\end{lemma}
\noindent\textit{Proof.}
     Since $\lambda(t,\alpha,\beta)\in\mathcal C^2(\mathbb R^d\times[0,1]^2,\mathbb C)$, $\nabla^2_t\lambda(0,\alpha,\beta)$ is continuous on $[0,1]^2$ and hence uniformly continuous. Therefore,
\begin{equation*}
\label{eq:multi-prod_of_lambda}
    \begin{aligned}
        \lim_{T\to\infty}\frac{1}{T}\sum_{k=0}^{\lfloor T\rfloor-1}\nabla_t^2\lambda\left(0,\frac{k}{T},\frac{k+1}{T}\right)
        &=-\lim_{T\to\infty}\frac{1}{T}\sum_{k=0}^{\lfloor T\rfloor-1}\nabla_t^2\lambda\left(0,\frac{k}{T},\frac{k}{T}\right)\\
        &=-\int_0^1\nabla^2_t\lambda(0,\alpha,\alpha)d\alpha\\
        &=-\Sigma\Sigma^*.
    \end{aligned}
\end{equation*}Since, for each $\alpha\in[0,1]$, $\nabla^2_t\lambda(0,\alpha,\alpha)$ is negative-definite, $\Sigma\Sigma^*$ is positive-definite. Finally, the product of the top eigenvalues at $t={\tau}/{\sqrt{T}}$ is computed:
    \begin{align*}
        \lim_{T\to\infty}\prod_{k=0}^{\lfloor T\rfloor-1}\lambda\left(\frac{\tau}{\sqrt{T}},\frac{k}{T},\frac{k+1}{T}\right)&=\lim_{T\to\infty}\prod_{k=0}^{\lfloor T\rfloor-1}\left(1+\frac{1}{2T}\tau^*\nabla^2_t\lambda(0,\frac{k}{T},\frac{k+1}{T})\tau+o\left(\frac{1}{T}\right)\right)\\
        &=\exp\left(\lim_{T\to\infty}\frac{1}{2T}\sum_{k=0}^{\lfloor T\rfloor-1}\tau^*\nabla_t^2\lambda\left(0,\frac{k}{T},\frac{k+1}{T}\right)\tau\right)\\
        &=\exp\left(-\frac{1}{2}\tau^*\Sigma\Sigma^*\tau\right).\tag*{\qed}
    \end{align*}
\section{Proof of the main result}
\label{sec:proof}
In this section, we prove the main result - Theorem~\ref{thm:main_result}. 
Let us first make some simple computations needed for the proof of the next lemma. Recall that $I$ is the neighborhood of the origin defined in Theorem~\ref{thm:multi-perturbation}.
\begin{enumerate}[(i)]
    \item Since $v(t,\alpha,\beta)\in\mathcal C^2(I\times[0,1]^2,\mathcal B)$, $\nabla_t^2v(0,\alpha,\beta)$ is uniformly continuous on $I\times[0,1]^2$. Thus, using the Taylor expansion for the first variable near zero, we have, for each $\tau\in\mathbb R^d$, all $0\leq k\leq\lt-2$, and all $T$ sufficiently large,
    \begin{equation}
        \label{eq:product_of_varphi_v}
    \begin{aligned}
        &v\left(\frac{\tau}{\sqrt T},\frac{k+1}{T},\frac{k+2}{T}\right)-v\left(\frac{\tau}{\sqrt T},\frac{k}{T},\frac{k+1}{T}\right)\\
        =&\ 1+\frac{\tau}{\sqrt{T}}\cdot\nabla_tv\left(0,\frac{k+1}{T},\frac{k+2}{T}\right)+\frac{1}{2T}\tau^*\nabla_t^2v\left(0,\frac{k+1}{T},\frac{k+2}{T}\right)\tau+o\left(\frac{1}{T}\right)\\
        &\ -1-\frac{\tau}{\sqrt{T}}\cdot\nabla_tv\left(0,\frac{k}{T},\frac{k+1}{T}\right)-\frac{1}{2T}\tau^*\nabla_t^2v\left(0,\frac{k}{T},\frac{k+1}{T}\right)\tau+o\left(\frac{1}{T}\right)\\
        =&\ o\left(\frac{1}{T}\right).
    \end{aligned}
    \end{equation}
    \item Since $\varphi(t,\alpha,\beta)$ is bounded on $I\times[0,1]^2$, for each $\tau\in\mathbb R^d$ and all $0\leq k\leq\lfloor T\rfloor-2$, we have by part 
(\romannumeral 3) in Theorem~\ref{thm:multi-perturbation} and \eqref{eq:product_of_varphi_v},
    \[
    \left\langle\varphi\left(\frac{\tau}{\sqrt T},\frac{k}{T},\frac{k+1}{T}\right),v\left(\frac{\tau}{\sqrt T},\frac{k+1}{T},\frac{k+2}{T}\right)\right\rangle=1+o\left(\frac{1}{T}\right).\] Therefore, for each $\tau\in\mathbb R^d$, uniformly in $0\leq i\leq\lfloor T\rfloor-2$,  
    \begin{equation}
        \label{eq:product_of_varphi_v_2}
    \lim_{T\to\infty}\prod_{k=i}^{\lfloor T\rfloor-2}\left\langle \varphi\left(\frac{\tau}{\sqrt T},\frac{k}{T},\frac{k+1}{T}\right),v\left(\frac{\tau}{\sqrt T},\frac{k+1}{T},\frac{k+2}{T}\right)\right\rangle=1.
    \end{equation}
    Moreover, by arguments similar to those leading to \eqref{eq:product_of_varphi_v_2}, the product in \eqref{eq:product_of_varphi_v_2} is uniformly bounded in $\tau/\sqrt{T}\in I$.
\item 
We showed in Lemma~\ref{lem:asymptotics} that $\nabla_t^2\lambda(0,\alpha,\alpha)$ is negative-definite for all $\alpha\in[0,1]$.
The absolute values of the eigenvalues of $\nabla_t^2\lambda(0,\alpha,\alpha)$ are bounded from below by a certain $c>0$ for all $\alpha\in[0,1]$. 
Then, by the continuity of $\nabla_t^2\lambda(0,\alpha,\beta)$, by making $I$ smaller (if needed), we can make sure that for all $T$ sufficiently large, all $\tau/\sqrt{T}\in I$, and all $0\leq k\leq\lt-2$,
\begin{align*}
    \lambda\left(\frac{\tau}{\sqrt{T}},\frac{k}{T},\frac{k+1}{T}\right)\leq 1-\frac{c|\tau|^2}{3T}\leq\exp\left(-\frac{c|\tau|^2}{3T}\right).
\end{align*}
Then, for all $\tau/\sqrt{T}\in I$ and $T$ sufficiently large,
\begin{equation}
\label{eq:dct_lambda}
    \prod_{k=0}^{\lfloor T\rfloor-1}\lambda\left(\frac{\tau}{\sqrt{T}},\frac{k}{T},\frac{k+1}{T}\right)\leq \exp\left(-\frac{c|\tau|^2}{4}\right).
\end{equation}
Combining this with the last statement in part (\romannumeral2), we see that there exists $C_2>0$ such that for all $\tau/\sqrt{T}\in I$, $0\leq k\leq\lt-2$, and $T$ sufficiently large,
\begin{equation}
\label{eq:dct}
    \prod_{k=0}^{\lfloor T\rfloor-1}\lambda\left(\frac{\tau}{\sqrt{T}},\frac{k}{T},\frac{k+1}{T}\right)\prod_{k=0}^{\lfloor T\rfloor-2}\left\langle \varphi\left(\frac{\tau}{\sqrt{T}},\frac{k}{T},\frac{k+1}{T}\right),v\left(\frac{\tau}{\sqrt{T}},\frac{k+1}{T},\frac{k+2}{T}\right)\right\rangle\leq C_2\exp\left(-\frac{c|\tau|^2}{4}\right).
\end{equation}
\end{enumerate}

Define two (possibly complex) measures on $\mathbb R^d$ for $u\in\mathbb R^d$, $f\in\mathcal B$, and $\mu\in\mathcal B_p'$:
\begin{equation}
    \label{def:measures}
    \begin{aligned}
        m_T^{u,f,\mu}(B)&={\mathrm{det}(\Sigma)}(2\pi T)^{d/2}\E_\mu[f(X_{T})1_B(S_T-u)],\\
        \Tilde m_T^{u,f}(B)&=e^{-\frac{1}{2T}u^*(\Sigma\Sigma^*)^{-1}u}\langle\nu,f\rangle \mathfrak L(B),
    \end{aligned}
\end{equation}
    where $\Sigma$ is defined in Lemma~\ref{lem:product_of_eigenvalues} and $\mathfrak L$ is the Lebesgue measure. We need to prove that they are sufficiently close as $T$ tends to $\infty$. 
    The following lemma is the main step towards this goal. 

\begin{lemma}
\label{lem:main_proof}
    For each $h$ that is a continuous integrable function on $\mathbb R^d$ with compactly supported Fourier transform and each constant $C>0$, we have, uniformly in $u\in\mathbb R^d$, $f\in\mathcal B_C$, and initial distribution $\mu\in\mathcal B'_p$,
    \[|\langle m_T^{u,f,\mu},h\rangle-\langle\Tilde m_T^{u,f},h\rangle|\to0 \text{, as $T\to\infty$}.\]
\end{lemma}
\noindent\textit{Proof.}
    Recall the expression in formula \eqref{eq:nagaev}.
    By Proposition~\ref{prop:multi-operator_product}, there exists $M>0$ such that, for all $t\in I$ and all $T$ sufficiently large,
        \begin{align*}
            &\prod_{k=0}^{\lt-1}Q\left(t,\frac{k}{T},\frac{k+1}{T}\right)\Tilde Q(t,T)f\\
            &=\prod_{k=0}^{\lt-1}\lambda\left(t,\frac{k}{T},\frac{k+1}{T}\right)\prod_{k=0}^{\lt-2}\left\langle\mu\left(t,\frac{k}{T},\frac{k+1}{T}\right),v\left(t,\frac{k+1}{T},\frac{k+2}{T}\right)\right\rangle\\
            &\quad\quad\cdot\left\langle \varphi\left(t,\frac{\lt-1}{T},\frac{\lt}{T}\right),\Tilde Q(t,T)f\right\rangle\cdot v\left(t,0,\frac{1}{T}\right)\\
            &\quad+\prod_{k=0}^{\lt-1}\lambda\left(t,\frac{k}{T},\frac{k+1}{T}\right)\|{\tilde Q(t,T)}f\|p_{0,T}^t \cdot h\left(t,0,\frac{1}{T}\right)\\
            &\quad+\prod_{k=0}^{\lt-1}\lambda\left(t,\frac{k}{T},\frac{k+1}{T}\right)\|{\tilde Q(t,T)}f\|q_{0,T}^t \cdot v\left(t,0,\frac{1}{T}\right),
        \end{align*}where $|p_{0,T}^t|,|q_{0,T}^t|<M/T$.
    By the formula of inverse Fourier transformation, with $\hat h$ denoting the Fourier transform of $h$,
    \begin{align*}
        \langle m_T^{u,f,\mu},h\rangle&={\mathrm{det}(\Sigma)}(2\pi T)^{d/2}\E_\mu[f(X_{T})h(S_T-u)]\\
        &={\mathrm{det}(\Sigma)}\left({\frac{T}{2\pi}}\right)^{d/2}\E_\mu\left[f(X_{T})\int_{\mathbb R^d}e^{it\cdot (S_T-u)}\hat h(t)dt\right]\\
        &={\mathrm{det}(\Sigma)}\left({\frac{T}{2\pi}}\right)^{d/2}\int_{\mathbb R^d}\hat h(t)e^{-it\cdot u}\left\langle\mu, \prod_{k=0}^{\lt-1}Q\left(t,\frac{k}{T},\frac{k+1}{T}\right)\Tilde Q(t,T)f\right\rangle dt\\
        &=:\mathcal A_1+\mathcal A_2+\mathcal A_3,
    \end{align*}
    where 
    \begin{align*}
        \mathcal A_1={\mathrm{det}(\Sigma)}\left({\frac{T}{2\pi}}\right)^{d/2}&\int_{I}\hat h(t)e^{-it\cdot u}\prod_{k=0}^{\lt-1}\lambda\left(t,\frac{k}{T},\frac{k+1}{T}\right)\prod_{k=0}^{\lt-2}\left\langle\varphi\left(t,\frac{k}{T},\frac{k+1}{T}\right),v\left(t,\frac{k+1}{T},\frac{k+2}{T}\right)\right\rangle\\
        &\quad \cdot\left\langle\varphi\left(t,\frac{\lt-1}{T},\frac{\lt}{T}\right),\Tilde Q(t,T)f\right\rangle \left\langle\mu,v\left(t,0,\frac{1}{T}\right)\right\rangle dt,\\
        \mathcal A_2={\mathrm{det}(\Sigma)}\left({\frac{T}{2\pi}}\right)^{d/2}&\int_{I}\hat h(t)e^{-it\cdot u}\prod_{k=0}^{\lt-1}\lambda\left(t,\frac{k}{T},\frac{k+1}{T}\right)\|{\tilde Q(t,T)}f\|\\
        &\quad\cdot \left(p_{0,T}^t \left\langle\mu,v\left(t,0,\frac{1}{T}\right)\right\rangle+q_{0,T}^t \left\langle\mu,h\left(t,0,\frac{1}{T}\right)\right\rangle\right) dt,\\
        \mathcal A_3={\mathrm{det}(\Sigma)}\left({\frac{T}{2\pi}}\right)^{d/2}&\int_{I^c\cap K}\hat h(t)e^{-it\cdot u}\langle\mu, \prod_{k=0}^{\lt-1}Q\left(t,\frac{k}{T},\frac{k+1}{T}\right)\Tilde Q(t,T)f\rangle dt,
    \end{align*}
    and $K$ is the support of $\hat h$. By the change of variable with $t=\frac{\tau}{\sqrt{T}}$,
    \begin{align*}
        \mathcal A_1&=\int_{\mathbb R^d}\kappa_T^{\mu}(\tau)e^{-i\frac{\tau\cdot u}{\sqrt{T}}}\left\langle 1_I(\frac{\tau}{\sqrt{T}})\varphi\left(\frac{\tau}{\sqrt{T}},\frac{\lt-1}{T},\frac{\lt}{T}\right),\tilde Q\left(\frac{\tau}{\sqrt{T}},T\right)f\right\rangle d\tau,
    \end{align*}
    where
    \begin{align*}
        \kappa_T^{\mu}(\tau)&={\mathrm{det}(\Sigma)}\left({\frac{1}{2\pi}}\right)^{d/2}\cdot 1_I\left(\frac{\tau}{\sqrt{T}}\right)\hat h\left(\frac{\tau}{\sqrt{T}}\right)\prod_{k=0}^{\lt-1}\lambda\left(\frac{\tau}{\sqrt{T}},\frac{k}{T},\frac{k+1}{T}\right)\\
        &\quad\quad\cdot\prod_{k=0}^{\lt-2}\left\langle \varphi\left(\frac{\tau}{\sqrt{T}},\frac{k}{T},\frac{k+1}{T}\right),v\left(\frac{\tau}{\sqrt{T}},\frac{k+1}{T},\frac{k+2}{T}\right)\right\rangle\cdot\left\langle\mu, v\left(\frac{\tau}{\sqrt{T}},0,\frac{1}{T}\right)\right\rangle.
    \end{align*}
Define
\[
\kappa(\tau)={\mathrm{det}(\Sigma)}\left(\frac{1}{{2\pi}}\right)^{d/2}\hat h(0)\exp(-\tau^*\Sigma\Sigma^*\tau/2).
\]
By Lemma~\ref{lem:product_of_eigenvalues} and \eqref{eq:product_of_varphi_v_2}, $\kappa_T^{\mu}(\tau) \to\kappa(\tau)$ as $T\to\infty$, for each $\tau$, uniformly in $\mu\in\mathcal B'_p$. By \eqref{eq:dct}, $|\kappa_T^{\mu}(\tau)|\leq C_3\sup_{t\in I}|\hat h(t)|\exp(-\frac{c|\tau|^2}{4})$ with a certain constant $C_3>0$. Hence, by Lebesgue's dominated convergence theorem, uniformly in $\mu\in\mathcal B_p'$,
\begin{equation}
\label{eq:kappa_converges}
    \int_{\mathbb R^d}|\kappa_T^{\mu}(\tau)-\kappa(\tau)|d\tau\to0.
\end{equation}
    Note that $\tilde Q^*(\frac{\tau}{\sqrt{T}},T)\varphi(\frac{\tau}{\sqrt{T}},\frac{\lt-1}{T},\frac{\lt}{T})\to\nu$ in $\mathcal B'$ as $T\to\infty$.
    Again, by Lebesgue's dominated convergence theorem, uniformly $\mu\in\mathcal B'_p$,
    \begin{equation}
    \label{eq:tildeQ_converges}
        \int_{\mathbb R^d}\kappa_T^\mu(\tau)\left\|1_I\left(\frac{\tau}{\sqrt{T}}\right)\tilde Q^*\left(\frac{\tau}{\sqrt{T}},T\right)\varphi\left(\frac{\tau}{\sqrt{T}},\frac{\lt-1}{T},\frac{\lt}{T}\right)-\nu\right\|d\tau\to0.
    \end{equation} 
    From the definition of $\Tilde m_T^{u,f}$ in \eqref{def:measures} and the definition of $\kappa(\tau)$, it follows that
    \begin{equation*}
        \langle\Tilde m_T^{u,f},h\rangle=\int_{\mathbb R^d}\kappa(\tau)e^{-i\frac{\tau\cdot u}{\sqrt{T}}}\langle\nu,f\rangle d\tau.
    \end{equation*}
    Therefore, for all $f\in \mathcal B_C$,
    \begin{align*}
        &|\mathcal A_1-\langle\Tilde m_T^{u,f},h\rangle|\\
        &\leq\left|\int_{\mathbb R^d}(\kappa_T^{\mu}(\tau)-\kappa(\tau))e^{-i\frac{\tau\cdot u}{\sqrt{T}}}\langle \nu,f\rangle d\tau\right|\\
        &\quad+\left|\int_{\mathbb R^d}\kappa_T^{\mu}(\tau) e^{-i\frac{\tau\cdot u}{\sqrt{n}}}\left(1_I\left(\frac{\tau}{\sqrt{T}}\right)\left\langle \varphi\left(\frac{\tau}{\sqrt{T}},\frac{\lt-1}{T},\frac{\lt}{T}\right),\tilde Q\left(\frac{\tau}{\sqrt{T}},T\right)f\right\rangle-\left\langle\nu,f\right\rangle\right) d\tau\right|\\
        &\leq C\int_{\mathbb R^d}|\kappa_T^{\mu}(\tau)-\kappa(\tau)|d\tau\\
        &\quad+C\left|\int_{\mathbb R^d}\kappa_T^{\mu}(\tau)\left\|1_I\left(\frac{\tau}{\sqrt{T}}\right)\tilde Q^*\left(\frac{\tau}{\sqrt{T}},T\right)\varphi\left(\frac{\tau}{\sqrt{T}},\frac{\lt-1}{T},\frac{\lt}{T}\right)-\nu\right\|d\tau\right|.
    \end{align*}So, by \eqref{eq:kappa_converges} and \eqref{eq:tildeQ_converges}, $|\mathcal A_1-\langle\Tilde m_T^{u,f},h\rangle|\to0$, as $T\to\infty$, uniformly in $u\in\mathbb R^d$,  $f\in\mathcal B_C$, and $\mu\in\mathcal B'_p$. Again, by Proposition~\ref{prop:multi-operator_product},
    \begin{align*}
        |\mathcal A_2|&\leq{\mathrm{det}(\Sigma)}\left({\frac{T}{2\pi}}\right)^{d/2}\int_I\frac{2M}{T}\left|\hat h(t)\right|\cdot\left|\prod_{k=0}^{\lt-1}\lambda\left(t,\frac{k}{T},\frac{k+1}{T}\right)\right|dt\cdot\sup_{t\in I}\| {\tilde Q(t,T)}\|\|\mu\|\|f\|\\
        &\leq C\int_{\mathbb R^d}\tilde\kappa_T(\tau)d\tau,
    \end{align*}
    where \[\tilde\kappa_T(\tau)=\frac{2M}{T}{\mathrm{det}(\Sigma)}\left({\frac{1}{2\pi}}\right)^{d/2}\cdot 1_I\left(\frac{\tau}{\sqrt{T}}\right)\left|\hat h\left(\frac{\tau}{\sqrt{T}}\right)\right|\cdot\left|\prod_{k=0}^{\lt-1}\lambda\left(\frac{\tau}{\sqrt{T}},\frac{k}{T},\frac{k+1}{T}\right)\right|,\] since $\|\mu\|,\| {\tilde Q(t,T)}\|\leq1$. 
    By \eqref{eq:dct_lambda} and Lebesgue's dominated convergence theorem, $\int_{\mathbb R^d}\tilde\kappa_T(\tau)d\tau\to0$.
    Finally, by Lemma~\ref{lem:exponential_convergence_away_from_0}, uniformly in $u\in\mathbb R^d$,  $f\in\mathcal B_C$, and $\mu\in\mathcal B'_p$, as $T\to\infty$,
    \begin{align*}
        |\mathcal A_3|&\leq{\mathrm{det}(\Sigma)}\left({\frac{T}{2\pi}}\right)^{d/2}\int_{I^c\cap K}|\hat h(t)|\cdot\left|\left\langle\mu, \prod_{k=0}^{\lt-1}Q\left(t,\frac{k}{T},\frac{k+1}{T}\right)\Tilde Q(t,T)f\right\rangle\right|dt\\
        &\leq{\mathrm{det}(\Sigma)}\cdot r_K^{\lt}\left({\frac{T}{2\pi}}\right)^{d/2}\int_{I^c\cap K}|\hat h(t)|dt\cdot \|f\|\\
        &\leq C{\mathrm{det}(\Sigma)}\cdot r_K^{\lt}\left({\frac{T}{2\pi}}\right)^{d/2}\int_{I^c\cap K}|\hat h(t)|dt\to0.\tag*{\qed}
    \end{align*}

The local limit theorem as in Theorem~\ref{thm:main_result} with fixed $f$, $\mu$, and $u$ follows from Lemma~\ref{lem:main_proof} by Theorem 10.7 in \cite{Breiman}. 
The following lemma enables us to state uniform results in Theorem~\ref{thm:main_result}. It is a multi-dimensional version of the Lemma IV.5 in \cite{HennionHerve}.
\begin{lemma}
\label{lem:multi-compact}
    Let $\chi$ be a parameter ranging over a set $\mathcal X$ and, for $T\geq 1$, let $m_T^\chi,\Tilde m_T^\chi$ be positive measures on $\mathbb R^d$ with respect to which $\rho(x)=1\wedge 1/|x|^2$ is integrable. Suppose that $\sup_{T}\sup_{\chi\in\mathcal X}\Tilde m_T^\chi(\rho(x))$ is finite, and, for each continuous integrable  function $h$ on $\mathbb R^d$ whose Fourier transform has compact support, \begin{equation*}
    \label{eq:multi_compact_convergence}
        \lim_{T\to\infty}\sup_{\chi\in\mathcal X}|\langle m_T^{\chi},h\rangle-\langle\Tilde m_T^{\chi},h\rangle|=0.
    \end{equation*} Then, for each compactly supported real-valued continuous function $g$, we have
    \begin{equation*}
        \lim_{T\to\infty}\sup_{\chi\in\mathcal X}|\langle m_T^{\chi},g\rangle-\langle\Tilde m_T^{\chi},g\rangle|=0.
    \end{equation*}
\end{lemma}

\begin{proof}[Proof of Theorem~\ref{thm:main_result}]
    Since, for each real-valued $f\in\mathcal B_C$, $f=f_+-f_-$, where $f_+=f\vee0$ and $f_-=-(f\wedge0)$, it suffices to consider all non-negative $f\in\mathcal B_C$.
    Lemma~\ref{lem:main_proof} verifies the assumptions in Lemma~\ref{lem:multi-compact} with the measures defined as in \eqref{def:measures} and the parameter set $\mathcal X$ being all $u\in\mathbb R^d$, non-negative $f\in\mathcal B_C$, and initial distribution $\mu\in\mathcal B'_p$. 
    Therefore, the uniform convergence follows from the two preceding lemmas.
\end{proof}

\begin{appendices}
\renewcommand{\thesection}{Appendix \Alph{section}}
\label{sec:appendix}
\section{Non-arithmetic condition and non-lattice condition}
It is clear that, under Assumptions \hyperlink{as:markov_process}{(A1)-(A3)}, \hyperlink{as:non-arithmetic}{(A4)} implies \hyperlink{as:non-lattice}{(NL)}. Here, we will show that they are actually equivalent under Assumptions \hyperlink{as:markov_process}{(A1)-(A3)}.
The proof is based on the results in \cite{MR2322704} stated below. To apply the results, we consider a discrete Markov chain defined using $X_t$. 
Namely, we consider a Markov chain $\mathcal X_k$, $k\geq 1$, defined as the path of the Markov process $X_t$ on the interval $[k-1,k]$, on the metric space $(\mathfrak X,\mathscr B(\mathfrak X))$, where $\mathfrak X$ denotes $\mathcal D([0,1])$. 
Let $\mathfrak P$ denote the Markov kernel of $\mathcal X_k$, $\mathfrak{B}$ be the Banach space of bounded $\mathscr B(\mathfrak X)$-measurable complex functions on $\mathfrak X$ equipped with the supremum norm: $\|F\|=\sup_{\mathfrak x\in\mathfrak X}|F(\mathfrak x)|$. 
To introduce a counterpart of $Q(t,\alpha,\alpha)$, we define $I_\alpha\in\mathfrak B$, by $I_\alpha(\mathfrak x)=\int_0^1 b(\alpha,\mathfrak x_s)ds$, and the operator $\mathfrak Q(t,\alpha):\mathfrak B\to\mathfrak B$ by
    \begin{equation*}
        \mathfrak Q(t,\alpha)F(\mathfrak x)=\int_{\mathfrak X}\exp(it\cdot I_\alpha(\mathfrak y))F(\mathfrak y)\mathfrak P(\mathfrak x,d\mathfrak y).
    \end{equation*}
    
(i) We first prove that, for each $0\leq\alpha\leq1$ and $t \not=0$, $r(\mathfrak Q(t,\alpha))<1$.
The following lemma is a reformulation of the results in Theorem~3.1, Corollary~3.1, and Lemma~3.3 in \cite{MR2322704}, adapted to our situation:
\begin{namedtheorem}[Lemma~A.1]
    If $\mathfrak B(\mathfrak X)$ is countably generated and the Markov kernel $\mathfrak P$ satisfies the Doeblin condition, 
    then either $r(\mathfrak Q(t,\alpha))=r(\mathfrak P)=1$ and $\mathfrak Q(t,\alpha)$ is quasi-compact, or $r(\mathfrak Q(t,\alpha))<r(\mathfrak P)=1$.
\end{namedtheorem}
It is not hard to verify $\mathfrak B(\mathfrak X)$ is countably generated since $\mathfrak X=\mathcal D([0,1])$, and the Doeblin condition is satisfied for the operator $\mathfrak P$ given Assumptions~\hyperlink{as:markov_process}{(A1)-(A2)}. So, in order to prove (i), it remains to show that the first situation in the statement, i.e., where $r(\mathfrak Q(t,\alpha))=r(\mathfrak P)=1$ and $\mathfrak Q(t,\alpha)$ is quasi-compact, is not possible. 

Suppose that $r(\mathfrak Q(t,\alpha))=1$ and $\mathfrak Q(t,\alpha)$ is quasi-compact. Then there exists $\hat F\in\mathfrak B$ with $\|\hat F\|=1$ and an eigenvalue $\hat\lambda$ with $|\hat\lambda|=1$ such that $\mathfrak Q(t,\alpha)\hat F=\hat\lambda\hat F$. Namely, for each $\mathfrak x\in\mathcal D([0,1])$,
\begin{align*}
    \hat\lambda\hat F(\mathfrak x)=\mathfrak Q(t,\alpha)\hat F&=\int_{\mathfrak X}\exp(it\cdot I_\alpha(\mathfrak y))F(\mathfrak y)\mathfrak P(\mathfrak x,d\mathfrak y)\\
    &=\E_{\mathfrak x_1}\left[\exp\left(it\cdot\int_0^1b(\alpha,X_s)ds\right)\hat F(X_{\cdot})\right]
\end{align*}by the Markov property of the process $\{X_s\}_{s\geq0}$, where $X_{\cdot}$ stands for the path of $X_s$ on $[0,1]$. Note that the right-hand side is a function of $\mathfrak x_1$. 
Hence, there exists $\hat f\in\mathcal B$ with $\|\hat f\|=1$ such that $\hat f(\mathfrak x_1)=\hat F(\mathfrak x)$, and, for each $x\in \bm X$, 
\begin{equation}
\label{eqa:eigen}
    \hat\lambda\hat f(x)=\E_{x}\left[\exp\left(it\cdot\int_0^1b(\alpha,X_s)ds\right)\hat f(X_1)\right].\tag{A.1}
\end{equation}

For each $\e>0$ small, we can find $x\in\bm X$ such that $|\hat f(x)|>1-\e$ and,  by Assumption~\hyperlink{as:mixing}{(A2)}, we can find $N>0$ such that $\mathrm{TV}(P(N,x,\cdot),\nu)<\e$.
By applying \eqref{eqa:eigen} $N$ times, we obtain:
\[
\hat\lambda^N\hat f(x)=\E_{x}\left[\exp\left(it\cdot\int_0^Nb(\alpha,X_s)ds\right)\hat f(X_N)\right].
\]
It follows, from how we chose $x$ and $N$, that
\begin{align*}
    1-\e\leq|\hat\lambda^N\hat f(x)|\leq \E_{x}|\hat f(X_N)|\leq\langle\nu,|\hat f|\rangle+2\e,
\end{align*}which implies that $\langle\nu,|\hat f|\rangle>1-3\e$. Since this holds for arbitrary $\e>0$ small, we conclude that $\langle\nu,|\hat f|\rangle=1$, i.e., $|\hat f|=1$, $\nu$-a.s.

Let us come back to \eqref{eqa:eigen}. Since $|\hat f|=1$, $\nu$-a.s., we can write $\hat f=\exp(iu)$, $\nu$-a.s., where $u$ is a real-valued bounded measurable function. Then, \eqref{eqa:eigen} implies that
\[
\left|\E_{\nu}\left[\exp\left(i\left(t\cdot\int_0^1b(\alpha,X_s)ds+u(X_1)-u(X_0)\right)\right)\right]\right|=1,
\]
which contradicts with the non-lattice condition \hyperlink{as:non-lattice}{(NL)}.

(ii) We now prove that, for each $0\leq\alpha\leq1$ and $t \not=0$, $r(Q(t,\alpha,\alpha))<1$. There exists $C>0$ such that, for each $f\in\mathcal B$ with $\|f\|=1$ and $n>0$,
\[
\|Q(t,\alpha,\alpha)^nf\|=\|\mathfrak Q(t,\alpha)^n F\|<Cr(\mathfrak Q(t,\alpha))^n,
\]where $F$ is defined by $F(\mathfrak x)=f(\mathfrak x_1)$. Therefore, $r(Q(t,\alpha,\alpha))\leq r(\mathfrak Q(t,\alpha))<1$, and we conclude that \hyperlink{as:non-lattice}{(NL)}, together with \hyperlink{as:markov_process}{(A1)-(A3)}, implies \hyperlink{as:non-arithmetic}{(A4)}.

\end{appendices}
\section*{Acknowledgements}
We are grateful to Dmitry Dolgopyat and Yeor Hafouta for helpful discussions. The authors were supported by the NSF grant DMS-2307377. Leonid Koralov was supported by the Simons Foundation, Gift ID: MP-TSM-00002743.

\printbibliography
\end{document}